\documentclass[11pt]{amsart}

\usepackage[utf8]{inputenc}                     
\usepackage[english]{babel}                     

\usepackage{amsthm}                             
\usepackage{amsmath,amscd}
\usepackage{amssymb}
\usepackage{mathtools}                          

\usepackage{verbatim}                           
\usepackage{fullpage}				
\usepackage{hyperref}
\usepackage{color}

\usepackage[pagewise]{lineno}

\theoremstyle{plain}
\newtheorem{theorem}{Theorem}[section]
\newtheorem*{theoremA}{Theorem A}
\newtheorem*{theoremB}{Theorem B}

\newtheorem{lemma}[theorem]{Lemma}
\newtheorem{corollary}[theorem]{Corollary}

\theoremstyle{definition}

\newtheorem*{definitionnonum}{Definition}

\theoremstyle{remark}
\newtheorem{remark}{Remark}

\theoremstyle{plain}

\numberwithin{equation}{section}

\newcommand{\mR}{\mathbb{R}}                    

\newcommand{\e}{\text{e}}                       
\newcommand{\eps}{\varepsilon}                  
\newcommand{\der}{\text{d}}

\newcommand{\abs}[1]{\lvert{#1} \rvert}          
\newcommand{\aabs}[1]{\left\lvert{#1} \right\rvert}
\newcommand{\absg}[1]{\lvert{#1} \rvert_g}
\newcommand{\absgs}[1]{\lvert{#1} \rvert_{\sasaki{g}}} 
\newcommand{\norm}[1]{\lVert{#1} \rVert}         
\newcommand{\anorm}[1]{\left\lVert{#1}\right\rVert} 

\newcommand{\br}[1]{\abr{#1}}                   
\newcommand{\abr}[1]{\langle#1 \rangle}        
\newcommand{\babr}[1]{\big\langle#1 \big\rangle}
\newcommand{\brg}[1]{\langle#1 \rangle_g}
\newcommand{\brgs}[1]{\langle#1 \rangle_{\sasaki{g}}}

\newcommand{\mdef}{\coloneqq}                   
\newcommand{\comp}{\circ}                       

\newcommand\restr[2]{{
  \left.\kern-\nulldelimiterspace
  #1 
  \vphantom{\big|} 
  \right|_{#2} 
}}

\def\vint_#1{\mathchoice%
      {\mathop{\kern0.2em\vrule width 0.6em height 0.69678ex depth -0.58065ex
              \kern-0.8em \intop}\nolimits_{\kern-0.4em#1}}%
      {\mathop{\kern0.1em\vrule width 0.5em height 0.69678ex depth -0.60387ex
              \kern-0.6em \intop}\nolimits_{#1}}%
      {\mathop{\kern0.1em\vrule width 0.5em height 0.69678ex depth -0.60387ex
              \kern-0.6em \intop}\nolimits_{#1}}%
      {\mathop{\kern0.1em\vrule width 0.5em height 0.69678ex depth -0.60387ex
              \kern-0.6em \intop}\nolimits_{#1}}}
\def\vintslides_#1{\mathchoice%
      {\mathop{\kern0.1em\vrule width 0.5em height 0.697ex depth -0.581ex
              \kern-0.6em \intop}\nolimits_{\kern-0.4em#1}}%
      {\mathop{\kern0.1em\vrule width 0.3em height 0.697ex depth -0.604ex
              \kern-0.4em \intop}\nolimits_{#1}}%
      {\mathop{\kern0.1em\vrule width 0.3em height 0.697ex depth -0.604ex
              \kern-0.4em \intop}\nolimits_{#1}}%
      {\mathop{\kern0.1em\vrule width 0.3em height 0.697ex depth -0.604ex
              \kern-0.4em \intop}\nolimits_{#1}}}

\newcommand{\aveint}[2]{\mathchoice%
      {\mathop{\kern0.2em\vrule width 0.6em height 0.69678ex depth -0.58065ex
              \kern-0.8em \intop}\nolimits_{\kern-0.45em#1}^{#2}}%
      {\mathop{\kern0.1em\vrule width 0.5em height 0.69678ex depth -0.60387ex
              \kern-0.6em \intop}\nolimits_{#1}^{#2}}%
      {\mathop{\kern0.1em\vrule width 0.5em height 0.69678ex depth -0.60387ex
              \kern-0.6em \intop}\nolimits_{#1}^{#2}}%
      {\mathop{\kern0.1em\vrule width 0.5em height 0.69678ex depth -0.60387ex
              \kern-0.6em \intop}\nolimits_{#1}^{#2}}}

\newcommand{\doo}{\partial}

\DeclareMathOperator{\diver}{div}

\DeclareMathOperator{\Vol}{Vol}                 

\newcommand{\lv}{\mathtt{v}}
\newcommand{\lh}{\mathtt{h}}
\newcommand{\ls}{\mathtt{s}}

\newcommand{\vnabla}{\overset{\lv}{\nabla}}
\newcommand{\hnabla}{\overset{\lh}{\nabla}}
\newcommand{\vbundle}{\mathcal{V}}
\newcommand{\hbundle}{\mathcal{H}}

\newcommand{\vdiver}{\overset{v}{\diver}}
\newcommand{\sasaki}[1]{{{#1}_{\ls}}}

\newcommand{\nablasm}{{\nabla_{SM}}}

\newcommand{\K}{\mathcal{K}}
\newcommand{\esc}[1]{\mathcal{E}_{{#1}}}
\newcommand{\hflow}[2]{\phi_{#1,#2}^{\lh}}
\newcommand{\vflow}[2]{\phi_{#1,#2}^{\lv}}
\newcommand{\valueat}[1]{\Big|_{#1}}
\newcommand{\cm}{K_\nabla} 
\newcommand{\ball}[2]{B_{{#1}} ({#2})}
\newcommand{\sphere}[2]{S_{{#1}} ({#2})}
\newcommand{\middletext}[1]{\quad\text{#1}\quad}
\newcommand{\setsep}{\,;}
\newcommand{\set}[2]{\{ {#1} \setsep {#2} \}}
\newcommand{\tf}[1]{S^{#1}}
\renewcommand{\abs}[1]{\left|#1\right|}

\title{Tensor tomography on Cartan-Hadamard manifolds}

\author{Jere Lehtonen}
\address{Department of Mathematics and Statistics, University of Jyv\"askyl\"a, P.O. BOX 35 (MaD) FI-40014 University of Jyv\"askyl\"a, Finland}
\email{jere.ta.lehtonen@jyu.fi}

\author{Jesse Railo}
\email{jesse.t.railo@jyu.fi}

\author{Mikko Salo}
\email{mikko.j.salo@jyu.fi}

\date{} 

\begin{document}

\begin{abstract}
We study the geodesic X-ray transform on Cartan-Hadamard manifolds, generalizing the X-ray transforms on Euclidean and hyperbolic spaces that arise in medical and seismic imaging. We prove solenoidal injectivity of this transform acting on functions
and tensor fields of any order. The functions are assumed to be exponentially
decaying if the sectional curvature is bounded, and polynomially decaying
if the sectional curvature decays at infinity. This work extends the results
of~\cite{Leh16} to dimensions $n \geq 3$ and to the case of tensor fields
of any order.
\end{abstract}

\maketitle

\section{Introduction}\label{sec:introduction}

\subsection{Motivation} 

This article considers the geodesic X-ray transform on noncompact Riemannian
manifolds. This transform encodes the integrals of a function $f$, where $f$
satisfies suitable decay conditions at infinity, over all geodesics.
In the case of Euclidean space the geodesic X-ray transform is just the usual
X-ray transform involving integrals over all lines, and in two dimensions it coincides with the Radon transform introduced in the seminal work of Radon in 1917~\cite{Rad17}.

X-ray and Radon type transforms in Euclidean space are widely used as mathematical models for medical and industrial imaging methods, such as CT, PET, SPECT and MRI (see \cite{Natterer}). In these applications one is interested in reconstructing unknown coefficients in a bounded region. However, it is often convenient to model the problems in terms of compactly supported functions in the noncompact space $\mR^n$, which makes it possible to use Fourier transform based methods for instance.

Another important class of imaging problems arises in geophysics, when determining interior properties of the Earth from acoustic scattering or earthquake measurements. In these problems one encounters X-ray transforms over general families of curves, which often correspond to geodesic curves of a sound speed profile within the Earth. Moreover, if the sound speed is anisotropic (depends on direction), then one needs to consider geodesic X-ray transforms of tensor fields \cite{Sha94}. A typical feature is that rays originating near the Earth surface eventually curve back to the surface. A simple mathematical model, which has been used as a first approximation for this behaviour, is to think of the domain as embedded in hyperbolic space $\mathbb{H}^n$ and to consider the geodesic X-ray transform in $\mathbb{H}^n$ \cite{Bal2005}. The hyperbolic geodesic X-ray transform also appears in Electrical Impedance Tomography in connection with the method of Barber and Brown \cite{Berenstein} and in partial data problems \cite{KenigSalo}.

Another setting where X-ray transforms on noncompact manifolds appear is inverse scattering theory (for instance in quantum mechanics, acoustics, or electromagnetics). The connection between scattering theory and Radon type transforms goes back at least to Lax and Phillips \cite{LaxPhillips}, and the X-ray transform of a scattering potential can be determined from measurements of the full scattering amplitude at high frequencies (see e.g.\ \cite{Weder2014}). The X-ray transforms that appear in these contexts are often Euclidean. However, in inverse scattering applications related to general relativity and black holes one encounters more general manifolds that resemble asymptotically hyperbolic ones \cite{JoshiSaBarreto}, and in recent results on phaseless inverse scattering problems more general geodesic X-ray transforms also arise (see \cite{Klibanov2017} and references therein). We remark that both in quantum mechanics and general relativity, the functions that one would like to reconstruct are often not compactly supported and thus it is important to deal with noncompact manifolds.

In this article we will study the invertibility of geodesic X-ray transforms on noncompact Riemannian manifolds. Our results will include Euclidean and hyperbolic space as special cases, but will apply to more general manifolds with nonpositive curvature (Cartan-Hadamard manifolds). This work also follows the long tradition of integral geometry problems as discussed for instance in \cite{GelfandGindikinGraev, Hel99, Hel13}. Here one of the main points is that our results apply to manifolds that do not need to have special symmetries (see the recent preprint \cite{GGSU17} for related results).

\subsection{Results}

For Euclidean or hyperbolic
space in dimensions $n \geq 2$, one has the following basic theorems on the
injectivity of this transform (see~\cite{Hel99},~\cite{Jen04},~\cite{Hel94}):

\begin{theoremA}
If $f$ is a continuous function in $\mR^n$ satisfying $|f(x)|
\leq {C(1+|x|)}^{-\eta}$ for some $\eta > 1$, and if $f$ integrates to zero
over all lines in $\mR^n$, then $f \equiv 0$.
\end{theoremA}

\begin{theoremB}
If $f$ is a continuous function in the hyperbolic space $\mathbb{H}^n$
satisfying $|f(x)| \leq C \e^{-d(x,o)}$, where $o \in \mathbb{H}^n$ is some
fixed point, and if $f$ integrates to zero over all geodesics in
$\mathbb{H}^n$, then $f \equiv 0$.
\end{theoremB}

We remark that some decay conditions for the function $f$ are required,
since there are examples of nontrivial functions in $\mathbb{R}^2$ which
decay like $|x|^{-2}$ on every line and whose X-ray
transform vanishes~\cite{Zal82},~\cite{Arm94}.
Related results on the invertibility of Radon type transforms on constant
curvature spaces or noncompact homogeneous spaces may be found
in~\cite{Hel99},~\cite{Hel13}.

The purpose of this article is to give analogues of the above theorems
on more general, not necessarily symmetric Riemannian manifolds.
We will work in the setting of Cartan-Hadamard manifolds,
i.e.\ complete simply connected Riemannian manifolds with nonpositive
sectional curvature. Euclidean and hyperbolic spaces are special
cases of Cartan-Hadamard manifolds, and further explicit examples are recalled
in Section~\ref{sec:examples}. It is well known that any Cartan-Hadamard manifold
is diffeomorphic to $\mathbb{R}^n$,
the exponential map at any point is a diffeomorphism,
and the map $x \mapsto {d(x,p)}^2$
is strictly convex for any $p \in M$ (see e.g.~\cite{Pet06}).

\begin{definitionnonum}
Let $(M,g)$ be a Cartan-Hadamard manifold, and fix a point $o \in M$. If $\eta > 0$, define the spaces of exponentially and polynomially decaying continuous functions by
\begin{align*}
E_{\eta}(M) &= \{ f \in C(M) \,;\, |f(x)| \leq C \e^{-\eta d(x,o)} \text{ for some $C > 0$} \}, \\
P_{\eta}(M) &= \{ f \in C(M) \,;\, |f(x)| \leq {C(1+d(x,o))}^{-\eta} \text{ for some $C > 0$} \}.
\end{align*}
Also define the spaces
\begin{align*}
E_{\eta}^1(M) &= \{ f \in C^1(M) \,;\, |f(x)| + |\nabla f(x)| \leq C \e^{-\eta d(x,o)} \text{ for some $C > 0$} \}, \\
P_{\eta}^1(M) &= \{ f \in C^1(M) \,;\, \text{$|f(x)| \leq {C(1+d(x,o))}^{-\eta}$ and } \\
 & \hspace{100pt} \text{$|\nabla f(x)| \leq {C(1+d(x,o))}^{-\eta-1}$ for some $C > 0$} \}.
\end{align*}
Here $\nabla = \nabla_g$ is the total covariant derivative in $(M,g)$ and
$|\,\cdot\,| = |\,\cdot \,|_g$ is the $g$-norm on tensors.
\end{definitionnonum}

It follows from Lemma~\ref{lemma_u_finite} that if $f \in P_{\eta}(M)$
for some $\eta > 1$, then the integral of $f$ over any maximal geodesic
in $M$ is finite. For such functions $f$ we may define the geodesic
X-ray transform $I_0 f$ of $f$ by
\[
  I_0 f(\gamma)
  = \int_{-\infty}^{\infty} f(\gamma(t)) \,dt, \qquad \text{$\gamma$ is a geodesic}.
\]
The inverse problem for the geodesic X-ray transform is to determine
$f$ from the knowledge of $I_0 f$. By linearity, uniqueness for this inverse
problem reduces to showing that $I_0 f = 0$ implies $f = 0$.

More generally, suppose that $f$ is a $C^1$-smooth
symmetric covariant $m$-tensor field
on $M$, written in local coordinates (using the Einstein summation convention)
as
\begin{equation*}
f = f_{j_1\dots j_m}(x) \,dx^{j_1} \otimes \dots \otimes dx^{j_m}.
\end{equation*}
We say that $f \in P_{\eta}(M)$ if $|f|_g \in P_{\eta}(M)$, and $f \in P_{\eta}^1(M)$ if $|f|_g \in P_{\eta}(M)$ and $|\nabla f|_g \in P_{\eta+1}(M)$, etc.
We recall that, in terms of local coordinates, 
\[
  \absg{f(x)} = \Big( g^{j_1 k_1}(x)\cdots g^{j_m k_m}(x) f_{j_1\dots j_m}(x)f_{k_1\dots k_m}(x)\Big)^{1/2}
\]
where $(g^{jk})$ is the inverse matrix of $(g_{jk})$.

Now if $f \in P_{\eta}(M)$ for some $\eta > 1$, then the geodesic X-ray transform $I_m f$ of $f$ is well defined by the formula
\[
I_m f(\gamma) = \int_{-\infty}^{\infty} f_{\gamma(t)}(\dot{\gamma}(t), \ldots, \dot{\gamma}(t)) \,dt, \qquad \text{$\gamma$ is a geodesic}.
\]
This transform always has a kernel when $m \geq 1$: if $h$ is
a symmetric $(m-1)$-tensor field satisfying $h \in P_{\eta}^1(M)$
for some $\eta > 0$, then $I_m (\sigma \nabla h) = 0$ where
$\sigma$ denotes symmetrization of a tensor field
(see Section~\ref{subsec:tensors}).
We say that $I_m$ is solenoidal injective
if $I_m f = 0$ implies $f = \sigma \nabla h$
for some $(m-1)$-tensor field $h$.

Our first theorem proves solenoidal injectivity of $I_m$
for any $m \geq 0$ on Cartan-Hadamard manifolds
with bounded sectional curvature, assuming exponential
decay of the tensor field and its first derivatives.
We will denote the sectional curvature of a two-plane
$\Pi \subset T_x M$ by $K_x(\Pi)$, and we write
$-K_0 \leq K \leq 0$ if $-K_0 \leq K_x(\Pi) \leq 0$
for all $x \in M$ and for all two-planes $\Pi \subset T_x M$.

\begin{theorem}\label{thm_main1}
Let $(M,g)$ be a Cartan-Hadamard manifold of dimension $n \geq 2$, and assume that
\[
-K_0 \leq K \leq 0, \qquad \text{for some $K_0 > 0$.}
\]
If $f$ is a symmetric $m$-tensor field in $E_{\eta}^1(M)$ for
some $\eta > \frac{n+1}{2}\sqrt{K_0}$, and if $I_m f = 0$, then
$f = \sigma \nabla h$ for some symmetric $(m-1)$-tensor field $h$ such that $h \in E_{\eta-\eps}(M)$ for any $\eps > 0$.
(If $m=0$, then $f \equiv 0$.)
\end{theorem}

The second theorem considers the case where the sectional curvature
decays polynomially at infinity, and proves solenoidal injectivity
if the tensor field and its first derivatives also decay polynomially. 

\begin{theorem}~\label{thm_main2}
Let $(M,g)$ be a Cartan-Hadamard manifold of dimension $n \geq 2$, and assume that the function
\[
\K(x) = \sup \, \{ |K_x(\Pi)| \,;\,
  \Pi \subset T_x M \ \mathrm{is\ a\ two\text{-}plane} \}
\]
satisfies $\K \in P_{\kappa}(M)$ for some $\kappa > 2$.
If $f$ is a symmetric $m$-tensor field in $P_{\eta}^1(M)$ for
some $\eta > \frac{n+2}{2}$, and if $I_m f = 0$, then $f = \sigma \nabla h$ for
some symmetric $(m-1)$-tensor field $h \in P_{\eta-1}(M)$.
(If $m=0$, then $f \equiv 0$.)
\end{theorem}

The second theorem is mostly of interest in two dimensions because of the following rigidity phenomenon: any manifold of dimension $\geq 3$ that satisfies the conditions of the theorem is isometric to Euclidean space~\cite{GW82}. See Section \ref{sec:examples} for a discussion. We will give the proof in any dimension since this may be useful in subsequent work.

We remark that Theorems~\ref{thm_main1}--\ref{thm_main2} correspond to
Theorems A and B above, but the manifolds considered in
Theorems~\ref{thm_main1}--\ref{thm_main2} can be much more general and
include many examples with nonconstant curvature (see Section~\ref{sec:examples}).
The results will be proved by using energy methods based on Pestov identities,
which have been studied extensively in the case of compact manifolds with
strictly convex boundary. We refer to~\cite{Muk77},~\cite{PS88},~\cite{Sha94},~\cite{Kn02},~\cite{PSU14}
for some earlier results.
In fact, Theorems~\ref{thm_main1}--\ref{thm_main2} can be viewed as
an extension of the tensor tomography results in~\cite{PS88}
from the case of compact nonpositively curved
manifolds with boundary to the case of certain noncompact manifolds.
We remark that one of the main points in our theorems is that the functions
and tensor fields are not compactly supported (indeed, the compactly supported
case would reduce to known results on compact manifolds with boundary).

More recently, the work~\cite{PSU13} gave a particularly simple derivation
of the basic Pestov identity for X-ray transforms
and proved solenoidal injectivity of $I_m$ on simple two-dimensional manifolds.
Some of these methods were extended to all dimensions in~\cite{PSU15}
and to the case of attenuated X-ray transforms in~\cite{GPSU16}.
Following some ideas in~\cite{PSU13},
the work~\cite{Leh16} proved versions of
Theorems~\ref{thm_main1}--\ref{thm_main2} for the case of two-dimensional
Cartan-Hadamard manifolds.

In this paper we combine the main ideas in~\cite{Leh16} with the methods
of~\cite{PSU15} and prove solenoidal injectivity results on Cartan-Hadamard
manifolds in any dimension $n \geq 2$. However, instead of using the Pestov
identity in its standard form (which requires two derivatives of the
functions involved), we will use a different argument
from~\cite{PSU15} related to the $L^2$ contraction property of
a Beurling transform on nonpositively curved manifolds.
This argument dates back to~\cite{GK80a,GK80}, it only involves first order
derivatives and immediately applies to tensor fields of arbitrary order.
The $C^1$ assumption in Theorems~\ref{thm_main1}--\ref{thm_main2} is due to
this method of proof, and the decay assumptions are related to the growth of
Jacobi fields. We mention that Theorems~\ref{thm_main1}--\ref{thm_main2} also
extend the two-dimensional results of~\cite{Leh16} by assuming slightly
weaker conditions.

This article is organized as follows. Section~\ref{sec:introduction} is the introduction, and Section~\ref{sec:examples}
contains examples of Cartan-Hadamard manifolds. In Section~\ref{sec:preliminaries} we review basic facts
related to geodesics on Cartan-Hadamard manifolds, geometry of the sphere
bundle and symmetric covariant tensors fields, following~\cite{Leh16}, \cite{PSU15}, \cite{DS11}.
Section~\ref{sec:estimates} collects some estimates concerning the
growth of Jacobi fields and related decay properties for solutions of transport
equations. Finally, Section~\ref{sec:proofs} includes the proofs of the
main theorems based on $L^2$ inequalities for Fourier coefficients.

\subsection*{Acknowledgements}
All authors were supported by the Academy of Finland (Finnish Centre of Excellence
in Inverse Problems Research, grant numbers 284715 and 309963), and J.L.\ and M.S.\ were also partly supported by the European Research Council under the European Union's Seventh Framework Programme (FP7/2007-2013) / ERC Starting Grant agreement no 307023.

\section{Examples of Cartan-Hadamard manifolds}\label{sec:examples}

In this section we recall some facts and examples related to Cartan-Hadamard manifolds.
Most of the details can be found
in~\cite{BO69},~\cite{KW74},~\cite{GW79},~\cite{GW82},~\cite{Pet06}.
We first discuss the case of two-dimensional manifolds,
which is quite different compared to manifolds of
higher dimensions.

\subsection{Dimension two}
Let $K \in C^\infty(\mathbb{R}^2)$.
A theorem of Kazdan and Warner~\cite{KW74} states that a necessary
and sufficient condition for existence of a complete Riemannian metric
on $\mathbb{R}^2$ with Gaussian curvature $K$ is
\begin{equation}\label{eq:KW74cond}
  \lim_{r\to \infty} \inf_{\abs{x}\geq r}
  K(x) \leq 0.
\end{equation}
This provides a wide class of Riemannian metrics satisfying
the assumptions of Theorem~\ref{thm_main1} in dimension two.
However, this does not directly give an example of a manifold satisfying
the assumptions of Theorem~\ref{thm_main2} since the
condition~\eqref{eq:KW74cond} is given with respect to the Euclidean metric of
$\mathbb{R}^2$.

Examples of manifolds satisfying
the assumptions of Theorem~\ref{thm_main2}
can be constructed using warped products.
Let $(r,\theta)$ be the polar coordinates in $\mathbb{R}^2$
and consider a warped product
\begin{equation}\label{eq:warped-metric}
  ds^2 = dr^2 + f^2(r)d\theta^2,
\end{equation}
where $f$ is a smooth function that is positive for $r > 0$ and satisfies $f(0) = 0$ and $f'(0) = 1$.
This is a Riemannian metric on $\mR^2$ having Gaussian curvature
\begin{equation}
  K(x) = -\frac{f''(\abs{x})}{f(\abs{x})},
\end{equation}
which depends only on the Euclidean distance $|x| \mdef r(x)$ to the origin.
We remark that distances to the origin in the Euclidean metric and
in the warped metric coincide.
It is shown in~\cite[Proposition 4.2]{GW79} that for every
$k \in C^\infty([0,\infty))$ with $k \leq 0$ there exists a unique
warped metric of the form~\eqref{eq:warped-metric} such that $k(|x|) = K(x)$.
Hence warped products provide many examples of two-dimensional
manifolds for which
$\mathcal{K}(x) \leq C(1+\abs{x})^{-\kappa}$ with $\kappa > 0$,
i.e.~$\mathcal{K} \in P_\kappa(M)$.

\subsection{Higher dimensions}
Warped products can also be used to construct examples of
higher dimensional Cartan-Hadamard manifolds satisfying the assumptions of
Theorem~\ref{thm_main1}, see e.g.~\cite{BO69}.

In the case of Theorem~\ref{thm_main2}
it turns out that the decay condition for curvature
is very restrictive in higher dimensions:
the only possible geometry is the Euclidean one.
This follows directly from a theorem
by Greene and Wu in~\cite{GW82}.
If $M$ is a Cartan-Hadamard manifold with $n = \dim(M) \geq 3$,
$k(s) = \sup \{\,\mathcal{K}(x)\setsep x\in M, d(x,o) = s\,\}$,
where $o$ is a fixed point, and one of the following holds:
\begin{enumerate}
  \item $n$ is odd and $\lim\inf_{s \to \infty} s^2k(s) \to 0$ or
  \item $n$ is even and $\int_0^\infty sk(s) \, \der s$ is finite,
\end{enumerate}
then $M$ is isometric to $\mathbb{R}^n$.

\section{Geometric facts}\label{sec:preliminaries}

Throughout this work we will assume
$(M,g)$ to be an $n$-dimensional Cartan-Hadamard manifold with $n \geq 2$
unless otherwise stated.
We also assume unit speed parametrization for geodesics.

In this section we collect some preliminary facts on geodesics on Cartan-Hadamard manifolds, derivatives on the unit tangent bundle and related Jacobi fields, and tensor fields. These facts will be used in the subsequent sections.

\subsection{Behaviour of geodesics}\label{subsec:geodesics}
By the Cartan-Hadamard theorem the exponential map $\exp_x$
is defined on all of $T_x M$ and is a diffeomorphism
for every $x \in M$.
Hence every pair of points can be joined by a unique geodesic. Let $SM = \{ (x,v) \in TM \,;\, \abs{v}=1 \}$ be the unit sphere bundle, and if $(x,v) \in SM$ denote by $\gamma_{x,v}$ 
the unique geodesic with $\gamma(0) = x$ and $\dot{\gamma}(0) = v$. The triangle inequality implies that 
\begin{equation}\label{eq:geodesic-distance}
  d_g(\gamma_{x,v}(t),o) \geq |t| - d_g(x,o)
\end{equation}
for all $t \in \mR, o \in M$. 

We say that a geodesic $\gamma$
is escaping with respect to the point $o$ if
the function $t \mapsto d_g(\gamma(t),o)$ is strictly
increasing on the interval $[0,\infty)$. The set of all such
geodesics is denoted by $\esc{o}$.
For $\gamma_{x,v} \in \esc{o}$ the triangle inequality gives 
\begin{equation}\label{eq:espacing-geodesic-distance}
  d_g(\gamma_{x,v}(t),o) \geq
  \begin{cases}
    d_g(x,o), &\text{if } 0 \leq t \leq 2 d_g(x,o),\\
    t - d_g(x,o), &\text{if } 2 d_g(x,o) < t.
  \end{cases}
\end{equation}
However, since $(M,g)$ is a Cartan-Hadamard manifold, Jacobi field estimates give a stronger bound. For $\gamma_{x,v} \in \esc{o}$ one has (see \cite[Corollary 4.8.5]{Jos08} or \cite[Section 6.3]{Pet06})
\begin{equation}\label{eq:espacing-geodesic-distance-stronger}
  d_g(\gamma_{x,v}(t),o) \geq \sqrt{d_g(x,o)^2 + t^2}, \qquad t \geq 0.
\end{equation}

The following lemma is proved in~\cite{Leh16} in two dimensions. The proof in higher dimensions is identical, but we include a short argument for completeness.

\begin{lemma}\label{lma:escaping-direction}
Suppose $o \in M$. At least one of the geodesics
$\gamma_{x,v}$ and $\gamma_{x,-v}$ is in $\esc{o}$.
\end{lemma}
\begin{proof}
Since $(M,g)$ is a Cartan-Hadamard manifold, the function $h(t) = d_g(\gamma_{x,v}(t),o)^2$ is strictly convex, $h'' > 0$, on $\mR$. If $h'(0) \geq 0$ then $\gamma_{x,v}$ is escaping, and if $h'(0) \leq 0$ then $\gamma_{x,-v}$ is escaping.
\end{proof}

\subsection{On the geometry of the unit tangent bundle}\label{subsec:tm}
We first briefly explain the splitting of the tangent bundle of $SM$
into horizontal and vertical bundles. Then we give a short discussion on geodesics of $SM$. Finally, we include a proof that $SM$ is complete when $M$ is.

\subsubsection{The structure of the tangent bundle}
The following discussion is based on~\cite{Pat99},~\cite{PSU15},
where these topics are considered in more detail. We denote by $\pi \colon TM \to M$ the usual base point map
$\pi(x,v) = x$.
The connection map $\cm \colon T(TM) \to TM$ of the Levi-Civita
connection $\nabla$ of $M$
is defined as follows. Let $\xi \in T_{x,v} TM$
and $c \colon (-\eps,\eps) \to TM$ be a curve such that
$\dot{c}(0) = \xi$. Write
$
  c(t) = (\gamma(t),Z(t)),
$
where $Z(t)$ is a vector field along the curve $\gamma$,
and define
\begin{equation*}
  \cm(\xi) \mdef D_t Z(0) \in T_x M.
\end{equation*}

The maps $\cm$ and $d\pi$ yield a splitting
\begin{equation}\label{eq:TTM-splitting}
  T_{x,v} TM = \tilde\hbundle(x,v) \oplus \tilde\vbundle(x,v)
\end{equation}
where $\tilde\hbundle(x,v) = \ker \cm$ is the horizontal bundle
and $\tilde\vbundle(x,v) = \ker d_{x,v} \pi$
is the vertical bundle. Both are $n$-dimensional subspaces
of $T_{x,v} TM$.

On $TM$ we define the Sasaki metric $\sasaki{g}$ by
\begin{equation*}
  \brgs{v,w}
  = \brg{\cm (v),\cm (w)} + \brg{d\pi (v), d\pi (w)},
\end{equation*}
which makes $(TM,\sasaki{g})$ a Riemannian manifold
of dimension $2n$. The maps $\cm \colon \tilde\vbundle(x,v) \to T_x M$ and $d\pi
\colon \tilde\hbundle(x,v) \to T_x M$ are linear isomorphisms.
Furthermore, the splitting~\eqref{eq:TTM-splitting} is orthogonal with respect to $g_s$.
Using the maps $\cm$ and $d\pi$, we will identify vectors in the horizontal and vertical
bundles with corresponding vectors on $T_x M$.

The unit sphere bundle $SM$ was defined as
\begin{equation*}
  SM \mdef \bigcup_{x \in M} S_x M, \qquad S_x M \mdef \set{(x,v) \in T_xM}{|v|_g = 1}.
\end{equation*}
We will equip $SM$ with the metric induced by the Sasaki metric on $TM$. The geodesic flow $\phi_t(x,v) \colon \mR \times SM \to SM$
is defined as
\begin{equation*}
  \phi_t(x,v) \mdef (\gamma_{x,v}(t),\dot\gamma_{x,v}(t)).
\end{equation*}
The associated vector field is called the geodesic vector field
and denoted by $X$.

For $SM$ we obtain an orthogonal splitting
\begin{equation}\label{eq:TSM-splitting}
  T_{x,v} SM = \mR X(x,v) \oplus \hbundle(x,v) \oplus \vbundle(x,v)
\end{equation}
where $\mR X \oplus \hbundle(x,v) = \tilde\hbundle(x,v)$ and
$\vbundle(x,v) = \ker d_{x,v}(\pi|_{SM})$.
Both $\hbundle(x,v)$ and $\vbundle(x,v)$
have dimension $n-1$ and can be canonically identified
with elements in the codimension one subspace
${\{ v \}}^\perp \subset T_x M$ via
$d\pi$ and $\cm$, respectively. We will freely use this identification.

Following~\cite{PSU15}, if $u \in C^1(SM)$, then the gradient $\nablasm u$
has the decomposition
\begin{equation*}
  \nablasm u = (Xu)X + \hnabla u + \vnabla u,
\end{equation*}
according to~\eqref{eq:TSM-splitting}. The quantities 
$\hnabla u$ and $\vnabla u$ are called the horizontal
and the vertical gradients, respectively.
It holds that $\brg{\vnabla u(x,v),v} = 0$ and $\brg{\hnabla u(x,v),v} = 0$
for all $(x,v) \in SM$.

As discussed in~\cite{PSU15}, on two-dimensional manifolds the horizontal and vertical
gradients reduce to the horizontal and vertical vector fields
$X_\perp$ and $V$ via
\begin{equation*}
  \hnabla u(x,v) = -(X_\perp u(x,v)) v^\perp
  \middletext{and}
  \vnabla u(x,v) = (Vu(x,v)) v^\perp
\end{equation*}
where $v^\perp$ is such that $\{ v, v^\perp \}$
is a positive orthonormal basis of $T_x M$.
In~\cite{Leh16} the flows associated with
$X_\perp$ and $V$ were used to derive estimates for
$X_\perp u$ and $Vu$. We will proceed in a similar manner in the higher dimensional case.

Let $(x,v) \in SM$ and $w \in S_x M, \,w \perp v$.
We define $\hflow{w}{t} \colon \mR \to SM$ by $\hflow{w}{t}(x,v)
= (\gamma_{x,w}(t),V(t))$,
where $V(t)$ is the parallel transport of $v$ along $\gamma_{x,w}$. It holds
that
\begin{equation}\label{eq:hflow-K-dpi}
  \cm \left(\frac{\der}{\der t} \hflow{w}{t}(x,v)\valueat{t=0}\right) = 0
  \middletext{and}
  d\pi \left(\frac{\der}{\der t} \hflow{w}{t}(x,v)\valueat{t=0}\right) = w.
\end{equation}
We define $\vflow{w}{t} \colon \mR \to SM$ by
$\vflow{w}{t}(x,v) = (x,(\cos t)v + (\sin t)w)$. It holds that
\begin{equation}\label{eq:vflow-K-dpi}
  \cm \left(\frac{\der}{\der t} \vflow{w}{t}(x,v)\valueat{t=0}\right) = w
  \middletext{and}
  d\pi \left(\frac{\der}{\der t} \vflow{w}{t}(x,v)\valueat{t=0}\right) = 0.
\end{equation}
The following lemma states the relation 
between $\hflow{w}{t}$ and $\vflow{w}{t}$
and the horizontal and the vertical gradients of a function.

\begin{lemma}\label{lma:sm-gradients-flows}
Suppose $u$ is differentiable at $(x,v) \in SM$. Fix $w \in S_x M, w \perp v$.
Then it holds that
\begin{equation*}
  \brg{\hnabla u(x,v),w} = \frac{\der}{\der t}u(\hflow{w}{t}(x,v))
  \valueat{t=0}
 \end{equation*}
and
\begin{equation*}
  \brg{\vnabla u(x,v),w} = \frac{\der}{\der t}u(\vflow{w}{t}(x,v))
  \valueat{t=0}.
\end{equation*}
\end{lemma}
\begin{proof}
Using the chain rule and the equations~\eqref{eq:hflow-K-dpi} we get
\begin{equation*}
  \frac{\der}{\der t}u(\hflow{w}{t}(x,v))
  \valueat{t=0}
  = \brgs{\nabla_{SM} u(\hflow{w}{t}(x,v)),
    \frac{\der}{\der t}\hflow{w}{t}(x,v)}\valueat{t=0}
  = \brg{\hnabla u(x,v), w}.
\end{equation*}
For $\vnabla$ we use the equations~\eqref{eq:vflow-K-dpi}
in a similar fashion.
\end{proof}

The maps $\hflow{w}{t}$ and $\vflow{w}{t}$ are related to normal Jacobi
fields along geodesics. We can define
\begin{equation*}
  J^{\lh}_w(t) \mdef  \frac{\der}{\der s}\pi\left(\phi_t(\hflow{w}{s}(x,v))\right)
  \valueat{s=0}
  =  d_{\phi_t(x,v)}\pi\left(\frac{\der}{\der s}
  \phi_t(\hflow{w}{s}(x,v))\valueat{s=0} \right).
\end{equation*}
Since $\Gamma(s,t) = \pi\left(\phi_t(\hflow{w}{s}(x,v))\right)$ is a variation of $\gamma_{x,v}$ along geodesics, $J^{\lh}_w(t)$ is a Jacobi field along $\gamma_{x,v}$. It has the initial conditions
$J^{\lh}_w(0) = w$ and $D_t J^{\lh}_w(0) = 0$ by the symmetry lemma (see e.g.~\cite{Lee97}).

Replacing $\hflow{w}{s}$ with $\vflow{w}{s}$
gives a Jacobi field $J^{\lv}_w(t)$ with
the initial conditions $J^{\lv}_w(t)(0) = 0$ and $D_t J^{\lv}_w(t)(0) = w$.
In the both cases the Jacobi field is normal because $\br{v,w}_g = 0$.

By the symmetry lemma 
\begin{equation*}
  \cm \left(\frac{\der}{\der s}
  \phi_t(\hflow{w}{s}(x,v))\valueat{s=0} \right)
  = D_s \doo_t \gamma_{\hflow{w}{s}(x,v)}(t)\valueat{s=0}
  = D_t \doo_s \gamma_{\hflow{w}{s}(x,v)}(t)\valueat{s=0}
  = D_t J^{\lh}_w(t).
\end{equation*}
From the definition of the Sasaki metric we then see that
\begin{equation*}
  \babr{\nablasm u(x,v) , \frac{\der}{\der s}
  \phi_t(\hflow{w}{s}(x,v))\valueat{s=0}}_{\sasaki{g}}
  = \babr{\hnabla u(x,v), J^{\lh}_w(t)}_g + \babr{\vnabla u(x,v), D_t J^{\lh}_w(t)}_g.
\end{equation*}
and
\begin{equation*}
  \babr{\nablasm u(x,v) , \frac{\der}{\der s}
  \phi_t(\vflow{w}{s}(x,v))\valueat{s=0}}_{\sasaki{g}}
  = \babr{\hnabla u(x,v), J^{\lv}_w(t)}_g + \babr{\vnabla u(x,v), D_t J^{\lv}_w(t)}_g.
\end{equation*}

\begin{remark}\label{rmk:diff_ae} The constructions in this subsection remain valid at a.e.\ $(x,v) \in SM$ if one assumes that $u$ is in the space $W_\text{loc}^{1,\infty}(SM)$. Functions in $W_\text{loc}^{1,\infty}(SM)$ are characterized as locally Lipschitz functions, and further by Rademacher's theorem, differentiable almost everywhere and weak gradients equal to gradients almost everywhere (see e.g. \cite[Chapters 5.8.2--5.8.3]{Eva98}).
\end{remark}

\subsubsection{Geodesics on the unit tangent bundle} Next we describe some facts related to geodesics on $SM$ (see e.g.~\cite{BBNV03} and references therein). Let $R(U,V)$ denote the Riemannian curvature tensor. A curve $\Gamma(t) = (x(t),V(t))$ on $SM$ is a geodesic if and only if
\begin{equation}
\label{eq:smgeod}
\left\{
\begin{aligned}
\nabla_{\dot{x}}\dot{x} &= -R(V,\nabla_{\dot{x}}V)\dot{x} \\
\nabla_{\dot{x}}\nabla_{\dot{x}}V &= -\abs{\nabla_{\dot{x}}V}_g^2 V, \quad \abs{\nabla_{\dot{x}}V}_g^2 \text{ is a constant along $x(t)$}
\end{aligned}
\right.
\end{equation} holds for every $t$ in the domain of $\Gamma$ (see \cite[Equations 5.2]{Sas62}). Given $(x,v) \in SM$, the horizontal lift of $w \in T_xM$ is denoted by $w^\lh$, i.e. the unique vector $w^\lh \in T_{x,v}(SM)$ such that $d(\pi|_{SM})(w^\lh) = w$ and $\cm(w^\lh) = 0$, and the vertical lift $w^\lv$ is defined similarly. Initial conditions for $x, \dot{x}, V$ and $\nabla_{\dot{x}}V$ at $t = 0$ with $g(V(0),\nabla_{\dot{x}(0)}V(0)) = 0$ and $\abs{V(0)}_g = 1$ determine a unique geodesic $\Gamma = (x,V)$, by (\ref{eq:smgeod}), which satisfies the initial conditions $\Gamma(0) = (x(0),V(0))$ and $\dot{\Gamma}(0) = \dot{x}(0)^\lh+(\nabla_{\dot{x}(0)}V(0))^\lv$ where the lifts are done with respect to $(x(0),V(0)) \in SM$. The geodesics of $SM$ are of the following three types:
\begin{enumerate}
\item If $\nabla_{\dot{x}(0)}V(0) = 0$, then $\Gamma$ is a parallel transport of $V(0)$ along the geodesic $x$ on $M$ (horizontal geodesics).
\item If $\dot{x}(0) = 0$, then $\Gamma$ is a great circle on the fibre $\pi^{-1}(x(0))$ and $x(t) = x(0)$ (vertical geodesics, in this case one interprets the system (\ref{eq:smgeod}) via $\nabla_{\dot{x}} = D_t$).
\item All the rest, i.e. solutions of (\ref{eq:smgeod}) with initial conditions $\dot{x}(0) \neq 0$ and $\nabla_{\dot{x}(0)}V(0) \neq 0$ (oblique geodesics).
\end{enumerate}
We state the following lemma for the sake of clarity.

\begin{lemma}\label{lem:flowsAreGeod} Fix $(x,v) \in SM$ and $w \in S_xM$, $w \,\bot\, v$. Then $\phi_t(x,v)$ and $\hflow{w}{t}(x,v)$ are horizontal unit speed geodesics and $\vflow{w}{t}(x,v)$ is a vertical unit speed geodesic with respect to $t$.
\end{lemma}
\begin{proof} The fact that $\phi_t(x,v)$ and $\hflow{w}{t}(x,v)$ are horizontal geodesics and $\vflow{w}{t}(x,v)$ is a vertical geodesic follows immediately from their definitions and the above discussion based on the system of differential equations (\ref{eq:smgeod}). The fact that $\phi_t(x,v)$, $\hflow{w}{t}(x,v)$ and $\vflow{w}{t}(x,v)$ are unit speed follows from the equations (\ref{eq:hflow-K-dpi}) and (\ref{eq:vflow-K-dpi}) and the definition of the Sasaki metric.
\end{proof}

Lemma \ref{lem:flowsAreGeod} allows us to derive the following formulas which are used in the proof of Lemma \ref{lma:uf-gradient-estimates}.

\begin{corollary}\label{cor:geodFlowDiff} Let $(x,v) \in SM$. Assume that $Y \in T_{x,v}(SM)$ has the decomposition \[Y = aX(x,v) + H + V, \quad H \in \hbundle(x,v),  V \in \vbundle(x,v), a \in \mR.\] Then 
\[\begin{split}
(D\phi_t)_{x,v}(aX(x,v)) &= aX(\phi_t(x,v)), \\
(D\phi_t)_{x,v}(H) &= |H|_{\sasaki{g}}\Big[(J_{w_\lh}^\lh(t))^\lh + (D_t J_{w_\lh}^\lh(t))^\lv\Big],
 \\
(D\phi_t)_{x,v}(V) &= |V|_{\sasaki{g}}\Big[(J_{w_\lv}^\lv(t))^\lh + (D_t J_{w_\lv}^\lv(t))^\lv\Big],
\end{split}\]
where $D\phi_t$ is the differential of $\phi_t$, $w_\lh = d\pi(H)/\abs{d\pi(H)}_g$ and $w_\lv = K_{\nabla}(V)/\abs{K_{\nabla}(V)}_g$. Moreover, $(D\phi_t)_{x,v}(X(x,v))$ is orthogonal to $(D\phi_t)_{x,v}(H)$ and $(D\phi_t)_{x,v}(V)$.
\end{corollary}
\begin{proof} Lemma~\ref{lem:flowsAreGeod} gives that $\phi_s(x,v)$, $\hflow{w_\lh}{s}(x,v)$ and $\vflow{w_\lv}{s}(x,v)$ are unit speed geodesics on $SM$. If $\Gamma(s) = \phi_s(x,v)$, then $\Gamma(s)$ is a unit speed geodesic on $SM$, $\dot{\Gamma}(0) = X(x,v)$, and 
\[
  (D\phi_t)_{x,v}(X(x,v))
  = D\phi_t(\dot{\Gamma}(0)) = (\phi_t \circ \Gamma)'(0) = X(\phi_t(x,v)).
\]
Moreover, using the unit speed geodesic $\Gamma(s) = \hflow{w_\lh}{s}(x,v)$ on $SM$,
and using the formulas after Lemma~\ref{lma:sm-gradients-flows},
gives
\[\begin{split}
  (D\phi_t)_{x,v}(H) &= D\phi_t(|H|_{\sasaki{g}}\dot{\Gamma}(0))
  = |H|_{\sasaki{g}}(\phi_t \circ \Gamma)'(0) \\
  &= |H|_{\sasaki{g}}\Big[(J_{w_\lh}^\lh(t))^\lh + (D_t J_{w_\lh}^\lh(t))^\lv\Big]
	\end{split}
\]
which is orthogonal to $X(\phi_t(x,v))$.
Finally, the unit speed geodesic $\Gamma(s) = \vflow{w_\lv}{s}(x,v)$ on $SM$ gives
\[\begin{split}
  (D\phi_t)_{x,v}(V) &= D\phi_t(|V|_{\sasaki{g}}\dot{\Gamma}(0))
  = |V|_{\sasaki{g}}(\phi_t \circ \Gamma)'(0) \\
  &= |V|_{\sasaki{g}}\Big[(J_{w_\lv}^\lv(t))^\lh + (D_t J_{w_\lv}^\lv(t))^\lv\Big]
	\end{split}
\]
which is also orthogonal to $X(\phi_t(x,v))$.
\end{proof}

\subsubsection{Completeness of the unit tangent bundle}

We will need the fact that $SM$ is complete when $M$ is complete. This need arises from theory of Sobolev spaces on manifolds (see Section \ref{sec:proofs}). We could not find a reference so a proof is included.

\begin{lemma}\label{lem:SMcomplete} Let $M$ be a complete Riemannian manifold with or without boundary. Then $SM$ is complete.
\end{lemma}
\begin{proof} Let $(y^{(j)})$ be a Cauchy sequence in $(SM,d_{\sasaki{g}})$. We show that it converges in the topology induced by $\sasaki{g}$.
The definition of the Sasaki metric implies that 
\[
  L_{\sasaki{g}}(\Gamma) \geq \int_0^\tau \abs{d\pi_{\Gamma(t)}(\dot{\Gamma}(t))}_g \,\der t
  = L_g(\pi \circ \Gamma) \geq d_g(\pi(\Gamma(0)),\pi(\Gamma(\tau)))
\]
where $\Gamma: [0,\tau] \to SM$ is any piecewise $C^1$-smooth curve. Hence
\begin{equation}\label{eq:sasaki-comparison}
  d_{\sasaki{g}}(a,b) \geq d_g(\pi(a), \pi(b))
\end{equation}
for all $a, b \in SM$. The above inequality implies that $(\pi(y^{(j)}))$ is a Cauchy sequence in $(M,g)$ and converges, say to $p \in M$, by completeness of $M$.

Consider a coordinate neighborhood $U$ of $p$ in $M$, so that $\pi^{-1}(U)$ is diffeomorphic to $U \times S^{n-1}$. Choose an open set $V$ and a compact set $K$ so that $p \in V \subset K \subset U$. Now $\pi^{-1}(K)$ is homeomorphic to $K \times S^{n-1}$ which is compact as a product of two compact sets.
Since $\pi(y^{(j)}) \to p$, there exists $N$ such that $\pi(y^{(j)}) \in V$ for all $j \geq N$, and this implies $y^{(j)} \in \pi^{-1}(K)$ for all $j \geq N$. Hence $(y^{(j)})$ has a limit in $(\pi^{-1}(K),d_{\sasaki{g}}|_{\pi^{-1}(K)})$ since it is a Cauchy sequence, and thus $(y^{(j)})$ converges also in $(SM,d_{\sasaki{g}})$.
\end{proof}

\subsection{Symmetric covariant tensor fields}\label{subsec:tensors}
We denote by $\tf{m}(M)$ the set of $C^1$-smooth symmetric covariant $m$-tensor fields and by $\tf{m}_x(M)$ the
symmetric covariant $m$-tensors at point $x$. Following \cite{DS11} (where more details are also given), 
we define the map $\lambda_x \colon \tf{m}_x(M) \to C^\infty(S_x M)$,
\begin{equation*}
  \lambda_x(f)(v) = f_x(v,\dots,v)
\end{equation*}
which is given in local coordinates by 
\begin{equation*}
  \lambda_x (f_{i_1 \dots i_m} dx^{i_1} \otimes
  \dots \otimes dx^{i_m})(v)
  = f_{i_1 \dots i_m}(x)v^{i_1}\dots v^{i_m}.
\end{equation*}
The map $\lambda$ smoothly depends on $x$ and hence we get
an embedding $\lambda \colon \tf{m}(M) \to C^1(SM)$.
The map $\lambda$ identifies symmetric \emph{trace-free} covariant
$m$-tensor fields with spherical harmonics 
(with respect to $v$) of degree $m$ on $SM$. More precisely, if $\tf{m}_x(M)$ and $C^\infty(S_x M)$ are endowed
with their usual $L^2$-inner products, then
$\lambda_x$ is an isomorphism, and even an isometry
up to a factor, from the set of trace-free symmetric $m$-tensors at $x$ onto the set of spherical harmonics (with respect to $v$) of degree $m$ on $S_x M$ (see \cite[Lemma 2.4 and subsequent remarks]{DS11}). We will use this identification and do not always write $\lambda$ explicitly.

The symmetrization of a tensor is defined by
\begin{equation*}
  \sigma(\omega_1 \otimes \cdots \otimes \omega_m)
  = \frac{1}{m!}
  \sum_{\pi \in \Pi_m} \omega_{\pi(1)}
   \otimes \cdots \otimes \omega_{\pi(m)},
\end{equation*}
where $\Pi_m$ is the permutation group of $\{1,\dots,m\}$.
From the above expression we see that
if a covariant $m$-tensor field $f$
is in $E_\eta^1(M)$ or $P_\eta^1(M)$
for some $\eta > 0$, then so is $\sigma f$ too.
Furthermore, for $f \in \tf{m}(M)$ one has 
\begin{equation}\label{eq:sn=X}
  \lambda(\sigma \nabla f) = X\lambda(f).
\end{equation}
It follows from the last identity and the fundamental theorem of calculus that if $f \in P^1_{\eta}(M)$ for some $\eta > 0$, then $I_m(\sigma \nabla f) = 0$. This shows that $I_m$ always has a nontrivial kernel for $m \geq 1$, as described in the introduction.

The next lemma states how the decay properties
of a tensor field carry over to functions on $SM$.
\begin{lemma}\label{lma:tensor-sm-gradients}
Suppose $f \in S^m(M)$ and $\eta > 0$.
\begin{enumerate}
\item[(a)]
If $f \in E_\eta^1(M)$, then
\begin{equation*}
  \sup_{v \in S_x M}\absg{Xf(x,v)} \in E_\eta(M),\quad
  \sup_{v \in S_x M}\absg{\hnabla f(x,v)} \in E_\eta(M)
  \middletext{and}
  \sup_{v \in S_x M}\absg{\vnabla f(x,v)} \in E_\eta(M).
\end{equation*}

\item[(b)]
If $f \in P_\eta^1(M)$, then
\begin{equation*}
  \sup_{v \in S_x M}\absg{Xf(x,v)} \in P_{\eta+1}(M),\quad
  \sup_{v \in S_x M}\absg{\hnabla f(x,v)} \in P_{\eta+1}(M)
  \middletext{and}
  \sup_{v \in S_x M}\absg{\vnabla f(x,v)} \in P_\eta(M).
\end{equation*}
\end{enumerate}
\end{lemma}
\begin{proof}
(a)
The result for $Xf$ follows from~\eqref{eq:sn=X}.
To prove the other statements
we take $x \in M$ and use local normal coordinates
$(x^1,\dots,x^n)$ centered
at $x$ and the associated coordinates
$(v^1,\dots,v^n)$ for $T_x M$.
In these coordinates
$f(x) = f_{i_1 \dots i_m}(x) \,dx^{i_1} \otimes \dots \otimes dx^{i_m}$
and
$\nabla f(x) = \doo_{x_j} f_{i_1 \dots i_m}(x)
\,dx^{j} \otimes dx^{i_1} \otimes \dots \otimes dx^{i_m}$.
We see that
\begin{equation*}
    \absg{f(x)} = {\left( \sum_{{i_1,\dots,i_m}}
    \abs{f_{i_1 \dots i_m}(x)}^2 \right)}^{1/2}
\middletext{and}
    \absg{\nabla f(x)} =
      {\left(\sum_{j,i_1,\dots,i_m} \abs{\doo_{x_j} f_{i_1 \dots i_m}(x)}^2 \right)}^{1/2}.
\end{equation*}

For $Xf, \hnabla f$ and $\vnabla f$ at $x$ we
have coordinate representations (see~\cite[Appendix A]{PSU15})
\begin{equation*}
  \begin{split}
    &Xf(x,v) = v^j \doo_{x_j} f,\\
    &\hnabla f(x,v) = \left(\doo^{x_j} f
    - (v^k \doo_{x_k}f)v^j \right)\doo_{x_j},\\
    &\vnabla f(x,v) = \doo^{v_j}f\doo_{x_j}.
  \end{split}
\end{equation*}
We get that
\begin{equation*}
  Xf(x,v)X(x,v) + \hnabla f(x,v) = \doo^{x_j} f \doo_{x_j}
  = \doo^{x_j} f_{i_1 \dots i_m}(x)
  v^{i_1} \dots v^{i_m}\doo_{x_j}
\end{equation*}
and, using the orthogonality of $Xf(x,v)X(x,v)$ and $\hnabla f(x,v)$ and the Cauchy-Schwarz inequality, 
\begin{equation*}
  \sup_{v \in S_x M} \absg{\hnabla f(x,v)}
  \leq
  {\left(\sum_{j,i_1,\dots,i_m} \abs{\doo_{x_j} f_{i_1 \dots i_m}(x)}^2 \right)}^{1/2} = \absg{\nabla f(x)}.
\end{equation*}
This implies that $  \sup_{v \in S_x M}\absg{\hnabla f(x,v)} \in E_\eta(M)$.

For $\vnabla f$, the identity $\partial_{v_j} v^k = \delta_j^k - v_j v^k$ (see \cite{PSU15}) implies that 
\begin{align*}
  \vnabla f(x,v)
  &= \sum_{j=1}^n (f_{j i_2 \dots i_m} v^{i_2} \dots v^{i_m} - f(x,v) v_j) \partial_{x_j} + \ldots + \sum_{j=1}^n (f_{i_1 \dots i_{m-1} j} v^{i_1} \dots v^{i_{m-1}} - f(x,v) v_j) \partial_{x_j} \\
  &= m \sum_{j=1}^n (f_{j i_2 \dots i_m} v^{i_2} \dots v^{i_m} - f(x,v) v_j) \partial_{x_j}
\end{align*}
Thus orthogonality and expanding the squares gives 
\begin{equation*}
  \absg{\vnabla f(x,v)}^2 = m^2 \sum_{j=1}^n \abs{f_{j i_2 \dots i_m}(x) v^{i_2} \dots v^{i_m}}^2
  \leq m^2 \sum_{i_1,\dots,i_m} \abs{f_{i_1 \dots i_m}(x)}^2 = m^2 \abs{f(x)}_g^2
\end{equation*}
which in turn implies that
$\sup_{v \in S_x M}\absg{\vnabla f(x,v)} \in E_\eta(M)$.
The proof for (b) is the same.
\end{proof}

\section{Growth estimates}\label{sec:estimates}

Throughout this section we assume that $f$ is a symmetric covariant $m$-tensor field in $P_\eta(M)$ for some $\eta > 1$. The main results in this section are Lemmas \ref{lma:uf-estimate} and \ref{lma:uf-gradient-estimates}. They state that if $f$ is such a tensor field, possibly with some additional decay at infinity, then the corresponding solution $u^f$ of the transport equation will have decay at infinity.

We begin by observing that the geodesic X-ray transform is well defined for such $f$.

\begin{lemma} \label{lemma_u_finite}
Let $f \in P_\eta(M)$ for some $\eta > 1$. For any $(x,v) \in SM$ one has 
\[
\int_{-\infty}^\infty \lvert f_{\gamma_{x,v}(t)}
  (\dot{\gamma}_{x,v}(t),\dots,\dot{\gamma}_{x,v}(t)) \rvert \, \der t < \infty.
\]
\end{lemma}
\begin{proof}
The assumption implies that $\lvert f_{\gamma_{x,v}(t)}
  (\dot{\gamma}_{x,v}(t),\dots,\dot{\gamma}_{x,v}(t)) \rvert \leq C(1+d(\gamma_{x,v}(t),o))^{-\eta}$. One can then change variables so that $t=0$ corresponds to the point on the geodesic that is closest to $o$, split the integral over $t \geq 0$ and $t \leq 0$, and use the fact that the integrands are $\leq C (1+\abs{t})^{-\eta}$ by the estimate \eqref{eq:espacing-geodesic-distance-stronger}.
\end{proof}

If $f \in P_{\eta}(M)$ for some $\eta > 1$, we may now define
\begin{equation*}
  u^f(x,v) \mdef \int_0^\infty f_{\gamma_{x,v}(t)}
  (\dot{\gamma}_{x,v}(t),\dots,\dot{\gamma}_{x,v}(t)) \, \der t.
\end{equation*}
It is straigthforward to see that
\begin{equation*}
  u^f(x,v) + (-1)^m u^f(x,-v) = I_m f(x,v)
\end{equation*}
for all $(x,v) \in SM$.

We have the usual reduction to the transport equation.

\begin{lemma}\label{lma:Xuf}
Let $f \in P_\eta(M)$ for some $\eta > 1$. Then $Xu^f = -f$.
\end{lemma}
\begin{proof}
By definition 
\begin{equation*}
  Xu^f(x,v) = \lim_{s \to 0} -\frac{1}{s}\int_0^s
  f_{\gamma_{x,v}(t)}(\dot\gamma_{x,v}(t),\dots,\dot\gamma_{x,v}(t)) \, \der t
  = -f_x(v,\dots,v).\qedhere
\end{equation*}
\end{proof}

Next we derive decay estimates for $u^f$ under the assumption that $I_mf = 0$.

\begin{lemma}\label{lma:uf-estimate}
Suppose that $I_m f = 0$.
\begin{enumerate}
  \item[(a)]
  If $f \in E_\eta(M)$ for $\eta > 0$, then
  \begin{equation*}
    \abs{u^f(x,v)} \leq C(1+d_g(x,o))\e^{-\eta d_g(x,o)}
  \end{equation*}
  for all $(x,v) \in SM$.
  \item[(b)]
  If $f \in P_\eta(M)$ for $\eta > 1$, then
  \begin{equation*}
    \abs{u^f(x,v)} \leq \frac{C}{{(1+d_g(x,o))}^{\eta-1}}
  \end{equation*}
  for all $(x,v) \in SM$.
\end{enumerate}
\end{lemma}
\begin{proof}
Since $I_mf = 0$, one has $\abs{u^f(x,v)} = \abs{u^f(x,-v)}$. By Lemma \ref{lma:escaping-direction}, possibly after replacing $(x,v)$ by $(x,-v)$, we may assume that $\gamma_{x,v}$ is escaping. We have
\begin{equation*}
  \abs{u^f(x,v)}
  = \aabs{\int_0^\infty f(\gamma_{x,v}(t))(\dot\gamma_{x,v}(t),\dots,\dot\gamma_{x,v}(t)) \, \der t}
  \leq \int_0^\infty \absg{f(\gamma_{x,v}(t))} \, \der t.
\end{equation*}
The rest of the proof is as in~\cite[Lemma 3.2]{Leh16}.
\end{proof}

\begin{lemma}\label{lma:sm-gradients-symmetry}
Let $f \in P_\eta(M)$ for some $\eta > 1$. If $I_m f = 0$ and $u^f$ is differentiable at $(x,v) \in SM$, then
\begin{equation*}
  \hnabla u^f(x,-v) = (-1)^{m-1} \hnabla u^f(x,v)
  \middletext{and}
  \vnabla u^f(x,-v) = (-1)^m \vnabla u^f(x,v).
\end{equation*}
\end{lemma}
\begin{proof}
From $I_mf = 0$ it follows that
\begin{equation*}
  u^f(x,v) + (-1)^m u^f(x,-v) = 0.
\end{equation*}
Fix $w \in S_x M, \,w \perp v$. We note that
\begin{equation*}
  u^f (\hflow{w}{s}(x,-v)) + (-1)^m u^f(\hflow{-w}{-s}(x,v)) = 0
\end{equation*}
and hence
\begin{equation*}
  \frac{\der}{\der s} u^f(\hflow{w}{s}(x,-v))\valueat{s=0}
  = -(-1)^m \frac{\der}{\der s} (u^f(\hflow{-w}{-s}(x,v)))\valueat{s=0}
  = (-1)^m \frac{\der}{\der s} (u^f(\hflow{-w}{s}(x,v)))\valueat{s=0}.
\end{equation*}
By Lemma \ref{lma:sm-gradients-flows} 
\begin{equation*}
  \br{\hnabla u^f(x,-v),w} = (-1)^m\br{\hnabla u^f(x,v),-w}
  = -(-1)^m \br{\hnabla u^f(x,v),w}.
\end{equation*}

For $\vnabla u^f$ we use that
\begin{equation*}
  u^f (\vflow{w}{s}(x,-v)) + (-1)^m u^f(\vflow{-w}{s}(x,v)) = 0
\end{equation*}
and by Lemma \ref{lma:sm-gradients-flows} we get that
\begin{equation*}
  \br{\vnabla u^f(x,-v),w} = (-1)^{m-1} \br{\vnabla u^f(x,v),-w}
  = (-1)^m \br{\vnabla u^f(x,v), w}.\qedhere
\end{equation*}
\end{proof}

We move on to prove growth estimates for Jacobi fields.
These estimates will be used to derive estimates for
$\hnabla u^f$ and $\vnabla u^f$.

\begin{lemma}\label{lma:jacobi-estimates}
Suppose $J(t)$ is a normal Jacobi field along a geodesic
$\gamma$.
\begin{enumerate}
\item[(a)]
If all sectional curvatures along $\gamma([0,\tau])$ are $\geq -K_0$
for some constant $K_0 > 0$, and if $J(0) = 0$ or $D_t J(0) = 0$, then
\begin{equation*}
  \absg{J(t)}
  \leq \absg{J(0)} \,\cosh(\sqrt{K_0} t)
  + \absg{D_t J(0)} \,\frac{\sinh(\sqrt{K_0} t)}{\sqrt{K_0}}
\end{equation*}
for $t \in [0,\tau]$.
\item[(b)]
If $t_0 \in (0,\tau)$, then
\begin{equation*}
\absg{D_t J(t)} + \absg{\frac{J(t)}{t} - D_t J(t)}
\leq \left[ \absg{D_t J(t_0)}
  + \abs{\frac{J(t_0)}{t_0} - D_t J(t_0)}_g \right]
  \,\e^{2 \int_{t_0}^t s \K(\gamma(s)) \,\der s}
\end{equation*}
for $t \in [t_0,\tau]$, where $\mathcal{K}$ is as defined in Theorem~\ref{thm_main2}.
\end{enumerate}
\end{lemma}
\begin{proof}
(a) follows from the Rauch comparison theorem \cite[Theorem 4.5.2]{Jos08}. For (b), we follow the argument in~\cite{Leh16}. Consider an orthonormal frame $\{ \dot{\gamma}(t), E_1(t), \ldots, E_{n-1}(t) \}$ obtained by parallel transporting an orthonormal basis of $T_{\gamma(0)} M$ along $\gamma$. Write $J(t) = u^j(t) E_j(t)$, so that the Jacobi equation becomes
\begin{equation} \label{eq:jacobi-system}
  \ddot{u}(t) + R(t) u(t) = 0
\end{equation}
where $u(t) = (u^1(t), \ldots, u^{n-1}(t))$
and $R_{jk} = R(E_j, \dot{\gamma}, \dot{\gamma}, E_k)$.
We wish to estimate $v(t) = \frac{u(t)}{t}$, and we do this by writing
$v(t) = A(t) + \frac{B(t)}{t}$ where
\begin{equation*}
  A(t) = \dot{u}(t), \qquad B(t) = u(t) - t \dot{u}(t).
\end{equation*}
By using the equation~\eqref{eq:jacobi-system}, we see that
\begin{align*}
  A(t) - A(t_0) &= - \int_{t_0}^t s R(s) v(s) \,\der s, \\
  B(t) - B(t_0) &= \int_{t_0}^t s^2 R(s) v(s) \,\der s.
\end{align*}
Write $g(t) = \abs{A(t)} + \abs{\frac{B(t)}{t}}$. If $t \geq t_0$ one has
\begin{equation*}
  g(t) = \abs{A(t_0) - \int_{t_0}^t s R(s) v(s) \,\der s}
  + \frac{1}{t} \abs{B(t_0)
  + \int_{t_0}^t s^2 R(s) v(s) \,\der s}
  \leq g(t_0) + 2 \int_{t_0}^t s \norm{R(s)} g(s) \,\der s.
\end{equation*}
The Gronwall inequality implies that
\begin{equation*}
  g(t) \leq g(t_0) \e^{2 \int_{t_0}^t s \norm{R(s)} \,\der s}.
\end{equation*}
The result follows from this, since $\norm{R(s)} = \sup_{\abs{\xi}=1} R(s) \xi \cdot \xi = \sup_{\dot{\gamma}(s) \in \Pi} K(\Pi) \leq \mathcal{K}(\gamma(s))$.
\end{proof}

\begin{corollary}\label{cor:jacobi-estimates}
Suppose that $(M,g)$ is a Cartan-Hadamard manifold. Let $\gamma$ be a geodesic
and $J$ a normal Jacobi field along
it, satisfying either $J(0) = 0$ and $\abs{D_t J(0)} \leq 1$ or $\abs{J(0)} \leq 1$ and $D_t J(0) = 0$.

\begin{enumerate}
\item[(a)]
If $-K_0 \leq K \leq 0$ and $K_0 > 0$, then
\begin{equation*}
  \absg{J(t)} \leq C \e^{\sqrt{K_0} t}
  \middletext{and}
  \absg{D_t J(t)} \leq C \e^{\sqrt{K_0} t}
\end{equation*}
for $t \geq 0$ where the constants
do not depend on the geodesic $\gamma$.

\item[(b)]
If
$\K \in P_\kappa(M)$ for some $\kappa > 2$,
then
\begin{equation*}
  \absg{J(t)} \leq C(t+1)
  \middletext{and}
  \absg{D_t J(t)} \leq C
\end{equation*}
for $t \geq 0$. If in addition $\gamma \in \esc{o}$, then
the constants do not depend on the geodesic $\gamma$.
\end{enumerate}
\end{corollary}

\begin{proof}
(a) The estimate for $\absg{J(t)}$ follows directly
from Lemma~\ref{lma:jacobi-estimates}.
Using the same notations as in the proof of that Lemma we have
$\absg{D_t J(t)} = |\dot{u}(t)|$ and
by integrating~\eqref{eq:jacobi-system} from
$0$ to $t$ we get
\begin{equation*}
\begin{split}
  |\dot{u}(t)|
  &\leq |\dot{u}(0)| + \int_0^t \norm{R(s)}|u(s)| \, \der s \\
  &\leq |D_t J(0)| + \int_0^t K_0 |J(s)| \, \der s\\
  &\leq C\e^{\sqrt{K_0} t}.
\end{split}
\end{equation*}

(b) 
For a fixed geodesic, the estimates follow
from Lemma~\ref{lma:jacobi-estimates}.
If $\K \in P_\kappa(M)$
for $\kappa > 2$, then
\begin{equation*}
  A \mdef \sup_{\gamma \in \esc{o}} \int_0^\infty s \K(\gamma(s)) \, \der s \leq C \sup_{\gamma \in \esc{o}} \int_0^\infty s (1+d_g(\gamma(s),o))^{-\kappa} \, \der s
  < \infty
\end{equation*}
by using \eqref{eq:espacing-geodesic-distance-stronger}. Let us fix $t_0 = 1$
and suppose that $J$ is a Jacobi field along a geodesic
in $\esc{o}$ whose initial values satisfy the given assumptions.
From Lemma~\ref{lma:jacobi-estimates} and (a) we then get
that
\begin{equation*}
\begin{split}
  \absg{J(t)}
  &\leq \e^{2A}
  \left( 2\absg{D_t J(1)}
    + \absg{J(1)} \right) t\\
  &\leq \e^{2A}C\e^{\sqrt{K_0}} t
\end{split}
\end{equation*}
for $t \geq 1$, where $K_0 = \sup_{x\in M} \K(x)$.

For $t \in [0,1]$ we can estimate $\absg{J(t)} \leq C\e^{\sqrt{K_0}}$.
By combining these two estimates we get
\begin{equation*}
  \absg{J(t)} \leq C(1+\e^{2A}t) \leq C\e^{2A}(1+t)
\end{equation*}
for $t \geq 0$, and the constants do not
depend on $\gamma \in \esc{o}$.

For $\absg{D_t J(t)}$, Lemma~\ref{lma:jacobi-estimates} gives
the estimate
\begin{equation*}
  \absg{D_t J(t)} \leq \e^{2A}
  \left( 2\absg{D_t J(1)}
    + \absg{J(1)} \right)
\end{equation*}
for $t \geq 1$, and for $t \in [0,1]$ we get a bound
from (a). Neither of these bounds depends on $\gamma \in \esc{o}$.
\end{proof}

\begin{lemma}\label{lma:uf-gradient-estimates}
Suppose that $I_m f = 0$.
\begin{enumerate}
\item[(a)]
If $-K_0 \leq K \leq 0$, $K_0 >0$ and $f \in E_\eta^1(M)$ for some $\eta > \sqrt{K_0}$,
then $u^f$ is differentiable along every geodesic on $SM$, $u^f \in W^{1,\infty}(SM)$ and
\begin{equation*}
  \abs{\hnabla u^f(x,v)}_g
  \leq C\e^{-(\eta-\sqrt{K_0})d_g(x,o)}
\end{equation*}
for a.e. $(x,v) \in SM$.
\item[(b)]
If $\K \in P_\kappa(M)$ for some $\kappa > 2$
and $f \in P_\eta^1(M)$ for some $\eta > 1$, then $u^f$ is differentiable along every geodesic on $SM$, $u^f \in W^{1,\infty}(SM)$ and
\begin{equation*}
  \abs{\hnabla u^f(x,v)}_g
  \leq \frac{C}{{(1 + d_g(x,o))}^{\eta-1}}
\end{equation*}
for a.e. $(x,v) \in SM$.
\end{enumerate}

The same estimates hold for $\vnabla u^f$
with the same assumptions.
\end{lemma}
\begin{proof}[Proof of $u^f \in W^{1,\infty}_\text{loc}(SM)$]
We show that $u^f$ is locally Lipschitz continuous. Fix $(x_0, v_0) \in SM$,
and suppose that $\Gamma(s)$ is a unit speed geodesic on $SM$ through $(x_0,v_0)$. We have for any $r > 0$
\begin{align}
\frac{u^f(\Gamma(r)) - u^f(\Gamma(0))}{r} &= \int_0^{\infty} \frac{f(\phi_t(\Gamma(r))) - f(\phi_t(\Gamma(0)))}{r} \,\der t \notag \\
 & = \int_0^{\infty} \frac{1}{r} \int_0^r \frac{\partial}{\partial s} \left[ f(\phi_t(\Gamma(s))) \right] \,\der s \,\der t \label{u_difference_quotient}\\
 &= \int_0^{\infty} \frac{1}{r} \int_0^r \brgs{\nabla_{SM} f(\phi_t(\Gamma(s))), D\phi_t(\Gamma(s)) \dot{\Gamma}(s)} \,\der s \,\der t. \notag
\end{align}

We write
\[
  \dot\Gamma(s) = \brgs{\dot{\Gamma}(s), X(\Gamma(s))}X(\Gamma(s)) + H_{\dot\Gamma}(s) + V_{\dot\Gamma}(s)
\]
where $H_{\dot\Gamma}(s) \in \hbundle(\Gamma(s))$ and $V_{\dot\Gamma}(s) \in \vbundle(\Gamma(s))$
When we apply Corollary \ref{cor:geodFlowDiff} to the right hand side of \eqref{u_difference_quotient} (and omit the identifications), we find that
\begin{equation}\label{eq:integral-on-rhs}
\begin{split}
 &\frac{u^f(\Gamma(r)) - u^f(\Gamma(0))}{r} 
 = \int_0^{\infty} \frac{1}{r} \int_0^r \Bigg[ Xf(\phi_t(\Gamma(s))) \brgs{\dot{\Gamma}(s), X(\Gamma(s))} \\
  &\quad+ \brg{\hnabla f(\phi_t(\Gamma(s))), \absgs{H_{\dot\Gamma}(s)} J_{w_\lh(s)}^\lh(t) + \absgs{V_{\dot\Gamma}(s)} J_{w_\lv(s)}^\lv(t)} \\
  &\quad+ \brg{\vnabla f(\phi_t(\Gamma  (s))), \absgs{H_{\dot\Gamma}(s)} D_t J_{w_\lh(s)}^\lh(t) + \absgs{V_{\dot\Gamma}(s)} D_t J_{w_\lv(s)}^\lv(t)} \Bigg] \,\der s \,\der t
\end{split}
\end{equation}
where $w_\lh(s) = H_{\dot\Gamma}(s)/\absgs{H_{\dot\Gamma}(s)}$ and $w_\lv(s) = V_{\dot\Gamma}(s)/\absgs{V_{\dot\Gamma}(s)}$.
Here the Jacobi fields are along the geodesic $\gamma_{\Gamma(s)}(t) := \pi(\phi_t(\Gamma(s)))$.
By definition their initial values fulfill the assumptions of Corollary~\ref{cor:jacobi-estimates}.

From this point on we will work under assumptions of (b).
The proof under assumptions of (a) is similar but simpler.
We fix a small $\eps > 0$. We show that the integral \eqref{eq:integral-on-rhs} has a uniform upper bound for every $r \in (0,1]$ and every geodesic $\Gamma$ through a point in $\ball{(x_0,v_0)}{\eps} \subset SM$. For $(x,v) \in SM$
we denote by $\mathcal{G}(x,v)$ the set of unit speed geodesics on $SM$ through $(x,v)$, and define \[J(x_0,v_0,\eps) \mdef \set{\Gamma \in \mathcal{G}(x,v)}{(x,v) \in \ball{(x_0,v_0)}{\eps}}.\]

For all $\Gamma \in J(x_0,v_0,\eps), \Gamma(0) = (x,v)$, and $s \in (0,r]$ the estimate~\eqref{eq:sasaki-comparison} gives that
$d_g(x,x_0) \leq \eps$ and
\[
  d_g(\gamma_{\Gamma(s)}(0),x) = d_g(\pi(\Gamma(s)),x)
  \leq d_{\sasaki{g}}(\Gamma(s),(x,v)) \leq s.
\]
The estimate~\eqref{eq:geodesic-distance} implies that
\begin{equation}\label{eq:dist-gamma-ineq}
\begin{split}
  d_g (\pi(\phi_t(\Gamma(s))),o)
  &= d_g(\gamma_{\Gamma(s)}(t),o) \geq t - d_g(\gamma_{\Gamma(s)}(0),x_0) \\
  &\geq t - \sup_{s\in(0,r]}d_g(\gamma_{\Gamma(s)}(0),o)
  \geq t - d_g(x,o) - r \\
  &\geq t - d_g(x_0,o) - \eps -r
\end{split}
\end{equation}
for all $t \geq t_0$ where $t_0 \mdef d_g(x_0,o) + r + \eps$.
We can use a trivial estimate $d_g (\pi(\phi_t(\Gamma(s))),o) \geq 0$ on the interval $[0,t_0]$. 
Further, the estimate~\eqref{eq:dist-gamma-ineq} gives
\begin{equation}\label{eq:K-decays-nicely}
 \K(\gamma_{\Gamma(s)}(t))
 \leq \frac{C}{(1+ d_g(\gamma_{\Gamma(s)}(t),o))^\eta}
 \leq \frac{C}{(1+ t - d_g(x_0,o)-\eps-r)^\eta}
\end{equation}
for all $t \geq t_0$ where the constant $C$ does not depend on $s \in (0,r]$ or 
the geodesic $\Gamma \in J(x_0,v_0,\eps)$, and hence
\begin{equation}\label{eq:Gamma-sup-K}
  \sup_{\substack{\Gamma \in J(x_0,v_0,\eps),\\s \in (0,r]}}
  \int_0^\infty t\K(\gamma_{\Gamma(s)}(t)) \, \der t < \infty.
\end{equation}

Using the proof of Corollary~\ref{cor:jacobi-estimates} together with~\eqref{eq:Gamma-sup-K},
we can find a constant $C$ which does not depend on $s \in (0,r]$ so that one has 
\begin{equation*}
  \absg{J_{w_\lh(s)}^\lh (t)} \leq Ct,\quad
  \absg{D_t J_{w_\lh(s)}^\lh (t)} \leq C
\end{equation*}
for all $t \geq 0$ and $\Gamma \in J(x_0,v_0,\eps)$.
Similar estimates hold also uniformly for $J_{w_\lv(s)}^\lv (t)$ and $D_t J_{w_\lv(s)}^\lv (t)$.

Recall that $\absgs{H_{\dot\Gamma}(s)}, \absgs{V_{\dot\Gamma}(s)} \leq \absgs{\dot{\Gamma}(s)} = 1$, and that $w_\lh(s), w_\lv(s)$ depend on $\Gamma$. By combining the above estimates for Jacobi fields with
estimate~\eqref{eq:dist-gamma-ineq}
and Lemma~\ref{lma:tensor-sm-gradients} we get
for the integrand in~\eqref{eq:integral-on-rhs} that
\begin{equation}
  \begin{split}
  & \big|Xf(\phi_t(\Gamma(s))) \brgs{\dot{\Gamma}(s), X(\Gamma(s))} \\
  &\quad+ \brg{\hnabla f(\phi_t(\Gamma(s))), \absgs{H_{\dot\Gamma}(s)} J_{w_\lh(s)}^\lh(t) + \absgs{V_{\dot\Gamma}(s)} J_{w_\lv(s)}^\lv(t)} \\
  &\quad+ \brg{\vnabla f(\phi_t(\Gamma  (s))), \absgs{H_{\dot\Gamma}(s)} D_t J_{w_\lh(s)}^\lh(t) + \absgs{V_{\dot\Gamma}(s)} D_t J_{w_\lv(s)}^\lv(t)}\big| \\
  &\leq \absg{Xf(\gamma_{\Gamma(s)}(t))}
  + \absg{\hnabla f(\gamma_{\Gamma(s)}(t))} \absg{\absgs{H_{\dot\Gamma}(s)}J_{w_\lh(s)}^\lh(t) + \absgs{V_{\dot\Gamma}(s)}J_{w_\lv(s)}^\lv(t)}\\
	&\quad+ \absg{\vnabla f(\gamma_{\Gamma(s)}(t))} \absg{\absgs{H_{\dot\Gamma}(s)}D_t J_{w_\lh(s)}^\lh(t) + \absgs{V_{\dot\Gamma}(s)}D_t J_{w_\lv(s)}^\lv(t)}\\
	&\leq \absg{Xf(\gamma_{\Gamma(s)}(t))}
  + \absg{\hnabla f(\gamma_{\Gamma(s)}(t))} \left(\absg{J_{w_\lh(s)}^\lh(t)} + \absg{J_{w_\lv(s)}^\lv(t)}\right) \\
  &\quad+ \absg{\vnabla f(\gamma_{\Gamma(s)}(t))} \left(\absg{D_t J_{w_\lh(s)}^\lh(t)} + \absg{D_t J_{w_\lv(s)}^\lv(t)}\right)\\
  &\leq \frac{Ct}{(1+t-d_g(x_0,o)- \eps - r)^{\eta + 1}} + \frac{C}{(1+t-d_g(x_0,o)- \eps - r)^{\eta}}
  \end{split}
\end{equation}
for all $t \in [t_0,\infty)$, $s \in (0,r]$ and $\Gamma \in J(x_0,v_0,\eps)$.
On the interval $[0,t_0]$ we also get a uniform upper bound since $f$, its covariant derivative and sectional curvatures
are all bounded.

We can conclude that integral on the right hand side
of \eqref{eq:integral-on-rhs} converges absolutely with some uniform bound $C < \infty$ over $r \in (0,1]$ and the set $J(x_0,v_0,\eps)$.
This shows that $u^f$ is locally Lipschitz, i.e.~$u^f \in W_\text{loc}^{1,\infty}(SM)$ (cf. Remark \ref{rmk:diff_ae}).
Moreover, the uniform estimate together with the dominated convergence theorem
guarantees that the limit $r\to 0$ of~\eqref{u_difference_quotient}
exists for all geodesics $\Gamma$ on $SM$. This finishes the first part of the proof.
\end{proof}
\begin{proof}[Proof of the gradient estimates] By Rademacher's theorem $u^f$ is differentiable almost everywhere, and thus we can assume that $u^f$ is differentiable at $(x,v) \in SM$.
By Lemmas~\ref{lma:escaping-direction}
and~\ref{lma:sm-gradients-symmetry} we can assume
that $(x,v)$ satisfies $\gamma = \gamma_{x,v} \in \esc{o}$.
We may also assume that $\hnabla u^f(x,v) \neq 0$. Since $\brg{\hnabla u^f(x,v),v} = 0$,
we can take $w = \hnabla u^f(x,v)/\absg{\hnabla u^f(x,v)}$ in
Lemma~\ref{lma:sm-gradients-flows} and get that
\begin{equation}
\begin{split}
  \absg{\hnabla u^f(x,v)}
  &= \frac{\der}{\der s} u^f(\hflow{w}{s}(x,v))\valueat{s=0}\\
  &= \int_0^\infty
  \brg{\hnabla f(\phi_t(x,v)), J^\lh(t)} + \brg{\vnabla f(\phi_t(x,v)), D_t J^\lh(t)}
  \, \der t
\end{split}
\end{equation}
where $J^\lh$ is again a Jacobi field along $\gamma$ fulfilling the 
assumptions of Corollary~\ref{cor:jacobi-estimates}. Under the conditions in part (a), the estimate~\eqref{eq:espacing-geodesic-distance-stronger} implies 
\[
\absg{\hnabla u^f(x,v)} \leq C \int_0^\infty
 \e^{-\eta d_g(\gamma(t),o)} \e^{\sqrt{K_0} \,t}
  \, \der t \leq \int_0^\infty
 \e^{-\eta \sqrt{d_g(x,o)^2 + t^2}} \e^{\sqrt{K_0} \,t}
  \, \der t.
\]
Writing $r = d_g(x,o)$ and splitting the integral over $[0,r)$ and $[r, \infty)$ gives 
\[
\absg{\hnabla u^f(x,v)} \leq C \left[ \int_0^r \e^{-\eta r} \e^{\sqrt{K_0} \,t} \,\der t+ \int_r^{\infty} \e^{-\eta t} \e^{\sqrt{K_0}\,t} \,\der t \right] \leq C \e^{-(\eta-\sqrt{K_0}) d_g(x,o)}.
\]
The above estimate also shows that $\absg{\hnabla u^f}$ is bounded.
Similarly, under the conditions in part (b), Lemma \ref{lma:tensor-sm-gradients}, Corollary \ref{cor:jacobi-estimates} and~\eqref{eq:espacing-geodesic-distance-stronger} imply 
\begin{align*}
\absg{\hnabla u^f(x,v)} &\leq C \int_0^{\infty} \frac{1+t}{(1+d_g(\gamma(t),o))^{\eta+1}} \,\der t + C \int_0^{\infty} \frac{C}{(1+d_g(\gamma(t),o))^{\eta}} \,\der t \\
 &\leq C \left[ \int_0^r \frac{1+t}{(1+r)^{\eta+1}} \,\der t + \int_r^{\infty} \frac{1+t}{(1+t)^{\eta+1}} \,\der t \right] \leq C(1+r)^{-(\eta-1)}
\end{align*}
where $r = d_g(x,o)$. The same arguments apply to $\vnabla u^f$. Hence $u^f \in W^{1,\infty}(SM)$ in the both cases, (a) and (b).
\end{proof}

\begin{lemma}\label{lma:sphere-volume}
\begin{enumerate}
  \item[(a)]
  If $-K_0 \leq K \leq 0$ and $K_0 > 0$, then
  \begin{equation*}
    \Vol \sphere{o}{r} \leq C  \e^{(n-1)\sqrt{K_0} r}
  \end{equation*}
  for all $r \geq 0$.

  \item[(b)]
  If $\K \in P_\kappa(M)$ for $\kappa >2$, then
  \begin{equation*}
    \Vol \sphere{o}{r} \leq Cr^{n-1}
  \end{equation*}
  for all $r \geq 0$.
\end{enumerate}
\end{lemma}
\begin{proof}
We define the mapping $f \colon S_o M \to \sphere{o}{r}$,
\begin{equation*}
  f(v) = (\pi \comp \phi_r)(o,v) = \exp_o(rv).
\end{equation*}
We denote by $\der \Sigma$ the volume form on $\sphere{o}{r}$
and have that
\begin{equation*}
  \Vol \sphere{o}{r} = \int_{\sphere{o}{r}} \der \Sigma
  = \int_{S_o M} f^*(\der \Sigma)
  = \int_{S_o M} \mu \, \der S,
\end{equation*}
where $\der S$ denotes the volume form on $S_o M$
(induced by Sasaki metric) and $\mu \colon S_o M \to \mR$.

Let $v \in S_o M$ and ${\{ w_i \}}_{i=1}^{n-1}$
be an orthonormal basis for $T_v S_o M$
with respect to Sasaki metric. By the Gauss lemma
${\{d_v f(w_i) \}}_{i=1}^{n-1}$ is an orthonormal basis
for $T_{f(v)} \sphere{o}{r}$ and
\begin{equation*}
  {f^*(\der\Sigma)}_v (w_1,\dots,w_{n-1})
  = \der\Sigma_{f(v)}(d_v f(w_1),\dots,d_v f(w_{n-1})).
\end{equation*}
It holds that $d_v f(w_i) = J_i(r)$ where
$J_i$ is a Jacobi field along the geodesic $\gamma_{o,v}$
with initial values $J_i(0) = d_v \pi(w_i)$
and $D_t J_i(0) = \cm(w_i)$. We get that
\begin{equation*}
  \abs{\mu(v)} \leq \prod_{i=1}^{n-1} \absg{d_v f(w_i)}
  = \prod_{i=1}^{n-1} \absg{J_i(r)}.
\end{equation*}

Since the tangent vectors
$w_i$ lie in $\vbundle(o,v)$ we have
$\absg{J_i(0)} = 0$ and $\absg{D_t J_i(0)} = \absgs{w_i} = 1$,
and the estimates for the volume of $\sphere{o}{r}$
then follow from Corollary~\ref{cor:jacobi-estimates}.
\end{proof}

\section{Proof of the main theorems}\label{sec:proofs}

In this section we will combine the facts above to prove Theorems \ref{thm_main1} and \ref{thm_main2}. We begin by introducing some useful notation related to operators on the sphere bundle and 
spherical harmonics.
One can find more details in~\cite{GK80},~\cite{DS11} and~\cite{PSU15}.
We prove the main theorems of this work in the end of this section.

The norm $\norm{\,\cdot\,}$ in this section will always be the $L^2(SM)$-norm.
We define the Sobolev space $H^1(SM)$ as the set of all $u \in L^2(SM)$ for which $\norm{u}_{H^1(SM)} < \infty$, where
\begin{equation*}
\begin{split}
  \norm{u}_{H^1(SM)}
  &= {\left(\norm{u}^2 + \norm{\nablasm u}^2\right)}^{1/2}\\
  &= {\left(\norm{u}^2 + \norm{Xu}^2 + \norm{\hnabla u}^2
    + \norm{\vnabla u}^2 \right)}^{1/2}.
\end{split}
\end{equation*}
Let $C_c^\infty(S M)$ denote the smooth compactly supported
functions on $S M$. It is well known that if $N$ is
complete Riemannian manifold, then $C_c^\infty(N)$ is dense in
$H^1(N)$ (see \cite[Satz 2.3]{Eic88}).
By Lemma~\ref{lem:SMcomplete} $SM$ is complete when $M$ is complete. Hence $C^{\infty}_c(SM)$ is dense in $H^1(SM)$.

For the following facts see~\cite{PSU15}. The vertical Laplacian
$\Delta : C^\infty(SM) \to C^\infty(SM)$
is defined as the operator 
\begin{equation*}
  \quad \Delta \mdef -\vdiver\vnabla.
\end{equation*}
Here $\vdiver$ denotes the vertical divergence which is the adjoint of $-\vnabla$ (see~\cite[Appendix A]{PSU15}). 
The Laplacian $\Delta$ has eigenvalues
 $\lambda_k = k(k+n-2)$ for $k =0,1,2,\dots$, and its eigenfunctions are homogeneous polynomials in $v$. One has an orthogonal eigenspace decomposition
\begin{equation*}
  L^2(SM) = \bigoplus_{k \geq 0} H_k(SM),
\end{equation*}
where $H_k(SM) \mdef \{ f \in L^2(SM) \setsep \Delta f = \lambda_k f \}$.
We define $\Omega_k = H_k(SM) \cap H^1(SM)$.
In particular, by Lemma \ref{lma:norm_sum} below any $u \in H^1(SM)$ can be written as
\begin{equation*}
  u = \sum_{k=0}^\infty u_k, \quad u_k \in \Omega_k,
\end{equation*}
where the series converges in $L^2(SM)$.

One can split the geodesic vector field in two parts,
$X = X_+ + X_-$, so that (by Lemma \ref{lma:norm_sum}) $X_+: \Omega_k \to H_{k+1}(SM)$
and $X_-: \Omega_k \to H_{k-1}(SM)$.
The next lemma gives an estimate for $X_{\pm} u$ in
terms of $Xu$ and $\hnabla u$.

\begin{lemma}\label{lma:norm_sum}
Suppose $u \in H^1(SM)$. Then $X_{\pm} u \in L^2(SM)$ and
\begin{equation*}
  \norm{X_+ u}^2 + \norm{X_- u}^2
  \leq
  \norm{Xu}^2 + \norm{\hnabla u}^2.
\end{equation*}
Moreover, for each $k \geq 0$ one has $u_k \in H^1(SM)$, and there is a sequence $(u_k^{(j)})_{j=1}^{\infty} \subset C^{\infty}_c(SM) \cap H_k(SM)$ with $u_k^{(j)} \to u_k$ in $H^1(SM)$ as $j \to \infty$.
\end{lemma}
\begin{proof}
Let $u \in C_c^\infty(SM)$. By~\cite[Lemma 4.4]{PSU15} one has the decomposition
\begin{equation*}
  \hnabla u = \vnabla \left[\sum_{l=1}^\infty \left( \frac{1}{l}X_+u_{l-1}
  - \frac{1}{l+n-2}X_-u_{l+1}\right) \right] + Z(u)
\end{equation*}
where $Z(u)$ is such that $\vdiver \,Z(u) = 0$. Hence
\begin{align*}
  \norm{\hnabla u}^2
  &= \sum_{l=1}^\infty \left(l(l+n-2)
  \anorm{\frac{1}{l}X_+ u_{l-1} - \frac{1}{l+n-2} X_- u_{l+1}} ^2\right)
  + \norm{Z(u)}^2\\
  &= \sum_{l=1}^\infty \left(\frac{l+n-2}{l}\norm{X_+ u_{l-1}}^2
  - 2 \br{X_+ u_{l-1},X_- u_{l+1}} + \frac{l}{l+n-2}\norm{X_- u_{l+1}}^2\right) + \norm{Z(u)}^2.
\end{align*}
We also have
\begin{align*}
  \norm{Xu}^2
  &= \norm{X_- u_1}^2
  + \sum_{l=1}^\infty \left(\norm{X_+ u_{l-1} + X_- u_{l+1}}^2\right) \\
  &= \norm{X_- u_1}^2
  + \sum_{l=1}^\infty \left(\norm{X_+ u_{l-1}}^2
    + 2 \br{X_+ u_{l-1},X_- u_{l+1}} +\norm{X_- u_{l+1}}^2\right)
\end{align*}
by the definition of $X_+$ and $X_-$. Adding up these estimates gives that
\begin{equation*}
\norm{Xu}^2 + \norm{\hnabla u}^2 = \norm{Z(u)}^2 + \norm{X_- u_1}^2 + \sum_{l=1}^\infty \left(A(n,l) \norm{X_+ u_{l-1}}^2
  + B(n,l) \norm{X_- u_{l+1}}^2 \right)
\end{equation*}
where $A(n,l) = 2+\frac{n-2}{l}$ and $B(n,l) = 1+\frac{l}{l+n-2}$. Since $A(n,l) \geq 1$ and $B(n,l) \geq 1$ for all $l = 1,2,\dots$ and $n \geq 2$, the estimate for $\norm{X_+ u}^2 + \norm{X_- u}^2$ follows when $u \in C^{\infty}_c(SM)$, and it extends to $H^1(SM)$ by density and completeness.

Moreover, if $u \in C^{\infty}_c(SM)$ and if $k \geq 0$, then the triangle inequality $\norm{X u_k} \leq \norm{X_+ u_k} + \norm{X_- u_k}$ and orthogonality imply that
\[
\norm{u_k} + \norm{X u_k} + \norm{\vnabla u_k} \leq \norm{u} + \norm{X_+ u} + \norm{X_- u} + \norm{\vnabla u}.
\]
We may also estimate $\hnabla u_k$ by \cite[Proposition 3.4]{PSU15} and orthogonality to obtain
\[
\norm{\hnabla u_k}^2 \leq (2k+n-1) \norm{X_+ u_k}^2 + (\sup_M K) \norm{\vnabla u_k}^2 \leq C_k (\norm{X_+ u}^2 + \norm{\vnabla u}^2).
\]
It follows from the first part of this lemma that
\[
\norm{u_k}_{H^1(SM)} \leq C_k \norm{u}_{H^1(SM)}, \qquad u \in C^{\infty}_c(SM).
\]
This extends to $u \in H^1(SM)$ by density and completeness. Finally, if $u \in H^1(SM)$ and the sequence $(u^{(j)}) \subset C^{\infty}_c(SM)$ satisfies $u^{(j)} \to u$ in $H^1(SM)$, then also $u^{(j)}_k \to u_k$ in $H^1(SM)$ by the above inequality.
\end{proof}

\begin{corollary}
\label{cor:chap4cor2}
Suppose $u \in H^1(SM)$. Then
\begin{equation*}
  \lim_{k\to\infty} \norm{X_+ u_k}_{L^2(SM)} = 0.
\end{equation*}
\end{corollary}
\begin{proof} By Lemma~\ref{lma:norm_sum} one has
\[
\norm{X_+ u}^2 = \sum_{k=0}^{\infty} \,\norm{X_+ u_k}^2 < \infty
\]
which implies the claim.
\end{proof}

\begin{lemma}\label{lma:iteration}
Let $u \in H^1(SM)$ and $k \geq 1$. Then one has that
\begin{equation*}
  \norm{X_- u_k} \leq D_n(k) \norm{X_+ u_k}
\end{equation*}
where
\begin{equation*}\begin{split}
 D_2(k) &= \begin{cases} \sqrt{2}, & k = 1 \\
1, & k \geq 2,\end{cases} \\
 D_3(k) &= {\left[1+\frac{1}{{(k+1)}^2(2k-1)}\right]}^{1/2} \\
 D_n(k) &\leq 1 \quad \text{for $n \geq 4$.}
\end{split}
\end{equation*}
\end{lemma}
\begin{proof}
This result was shown for smooth compactly supported functions
in~\cite[Lemma 5.1]{PSU15}.
The result follows for $u\in H^1(SM)$ by
an approximation argument using Lemma~\ref{lma:norm_sum}.
\end{proof}

The estimates from Section~\ref{sec:estimates}
allow us to prove the following result:
\begin{lemma}\label{lma:u-H1}
Suppose that $f$ is a symmetric $m$-tensor field and either of the following holds:
\begin{enumerate}
  \item[(a)] $-K_0 \leq K \leq 0$, $K_0 > 0$ and
  $f \in E_\eta^1(M)$ for $\eta > \frac{(n+1)\sqrt{K_0}}{2}$
  \item[(b)] $\K \in P_\kappa(M)$ for $\kappa > 2$ and
  $f \in P_\eta^1(M)$ for $\eta > \frac{n+2}{2}$.
\end{enumerate}
Then $u^f \in H^1(SM)$.
\end{lemma}
\begin{proof}
We prove only (a), the proof for (b) is similar.
By Lemma~\ref{lma:uf-gradient-estimates} we have that $u^f \in W^{1,\infty}(SM)$.
Lemma~\ref{lma:uf-estimate} gives that
\begin{equation*}
  |u^f(x,v)| \leq C(1+d_g(x,o))\e^{-\eta d_g(x,o)}
\end{equation*}
on $SM$. By using the coarea formula with Lemma~\ref{lma:sphere-volume} we get
\begin{equation*}
\begin{split}
  \int_{SM} |u^f(x,v)|^2 \, \der V_{\sasaki{g}}
  &\leq C \int_{M} (1+d_g(x,o))^2 \e^{-2\eta d_g(x,o)} \, \der V_g \\
  &= C\int_0^\infty (1+r)^2 \e^{-2\eta r}
  \left(\int_{\sphere{o}{r}} \, \der S\right) \, \der r \\
  & \leq C\int_0^\infty (1+r)^2 \e^{-2\eta r}\e^{(n-1)\sqrt{K_0}r} \der r.\\
\end{split}
\end{equation*}
The last integral above is finite and hence $u^f \in L^2(SM)$.
Similar calculations using Lemmas \ref{lma:Xuf} and \ref{lma:uf-gradient-estimates} show that $Xu^f, \hnabla u^f$ and $\vnabla u^f$
all have finite $L^2$-norms under the assumption $\eta > \frac{(n+1)\sqrt{K_0}}{2}$, and therefore the $H^1$-norm of $u^f$ is finite.
\end{proof}

We are ready to prove our main theorems.
\begin{proof}[Proof of Theorems~\ref{thm_main1} and~\ref{thm_main2}]
Suppose that the $m$-tensor field $f$ and the sectional curvature
$K$ satisfy the assumptions of Theorem~\ref{thm_main1} or~\ref{thm_main2}.
Recall that we identify $f$ with a function on $SM$ as described in Section \ref{subsec:tensors}. Then $u=u^f$ is in $H^1(SM)$ by Lemma~\ref{lma:u-H1},
and Lemma~\ref{lma:Xuf} states that $Xu = -f$ on $SM$. Note also that $f \in H^1(SM)$, which follows as in the proof of Lemma \ref{lma:u-H1}.

Since $f$ is of degree $m$ it has a decomposition
\begin{equation*}
  f = \sum_{k=0}^m f_k,\quad f_k \in \Omega_k,
\end{equation*}
and $u$ has a decomposition
\begin{equation*}
  u = \sum_{k=0}^\infty u_k, \quad u_k \in \Omega_k.
\end{equation*}

We first show that $u_k = 0$ for $k \geq m$.
From $Xu = -f$ it follows that for $k \geq m$ we have 
\begin{equation*}
  X_+ u_k + X_- u_{k+2} = 0.
\end{equation*}
This implies that
\begin{equation} \label{xplusminus_solution_inequality}
  \norm{X_+ u_k} \leq \norm{X_- u_{k+2}}, \qquad k \geq m.
\end{equation}

Fix $k \geq m$. We apply Lemma~\ref{lma:iteration} and the inequality \eqref{xplusminus_solution_inequality} iteratively to get
\begin{equation*}
\begin{split}
  \norm{X_- u_k}
  &\leq D_n(k) \norm{X_+ u_k} \\
  &\leq D_n(k) \norm{X_- u_{k+2}}\\
  &\leq D_n(k) D_n(k+2) \norm{X_+ u_{k+2}} \\
  &\leq \left[ \prod_{l=0}^N D_n(k+2l) \right] \norm{X_+ u_{k+2N}}.
\end{split}
\end{equation*}
By Corollary~\ref{cor:chap4cor2}
\begin{equation*}
  \lim_{l \to \infty }\norm{X_+ u_{k+2l}} = 0.
\end{equation*}
Moreover, as stated in \cite[Theorem 1.1]{PSU15}, one has 
\begin{equation*}
  \prod_{l=0}^\infty D_n(k+2l) < \infty.
\end{equation*}
Thus we obtain that
\begin{equation*}
  \norm{X_- u_k} = \norm{X_+ u_k} = 0.
\end{equation*}
This gives $Xu_k = 0$, which implies that
$t \mapsto u_k(\phi_t(x,v))$ is a constant function
on $\mR$ for any $(x,v) \in SM$. Since $u$ decays to zero
along any geodesic we must have $u_k = 0$, and this holds for all $k \geq m$.

It remains to verify that the equation $Xu = -f$ on $SM$ together with the fact $u = \sum_{k=0}^{m-1} u_k$ imply the conclusions of Theorems \ref{thm_main1} and \ref{thm_main2}. This is done as in \cite[end of Section 2]{PSU13}. Suppose that $m$ is odd (the case where $m$ is even is similar).
The function $f$ is a homogeneous polynomial
of order $m$ in $v$ and hence its Fourier decomposition has only odd terms, i.e.
\begin{equation*}
  f = f_m + f_{m-2} + \cdots + f_1.
\end{equation*}
It follows that the decomposition of $u$ has only even terms,
\begin{equation*}
  u =  u_{m-1} + u_{m-3} + \cdots + u_{0}.
\end{equation*}

By taking tensor products with the metric $g$
and symmetrizing it is possible to raise
the degree of a symmetric tensor:
if $F \in \tf{m}(M)$, then $\alpha F \mdef \sigma(F \otimes g) \in \tf{m+2}(M)$.
Functions $\lambda(\alpha F)$ and $\lambda(F)$ have the same
restriction to $SM$, since
$\lambda(g)$ has a constant value 1 on $SM$.

We define $h \in \tf{m-1}(M)$ by 
\begin{equation*}
h \mdef -\sum_{j=0}^{(m-1)/2} \alpha^j (U_{m-1-2j}),
\end{equation*}
where $U_{m-1-2j}(x)$ is the unique symmetric trace-free $(m-1-2j)$-tensor field which satisfies $\lambda_x(U_{m-1-2j}(x)) = u_{m-1-2j}(x,\,\cdot\,)$, see Section \ref{subsec:tensors}.

Then $\lambda(h) = -u$ on $SM$. 
Equation \eqref{eq:sn=X} gives $\lambda(\sigma \nabla h) = X(\lambda h) = -Xu = \lambda(f)$ on $SM$.
Since both $f$ and $\sigma \nabla h$ are symmetric we get $f = \sigma \nabla h$. To show the decay condition for $h$, we assume the conditions of Theorem \ref{thm_main1} and observe that Lemma \ref{lma:uf-estimate} implies that for any fixed $\eps > 0$, 
  \begin{equation} \label{u_estimate_epsilon}
    \abs{u(x,v)} \leq C(1+d_g(x,o))\e^{-\eta d_g(x,o)} \leq C_{\eps} \e^{-(\eta-\eps) d_g(x,o)}.
  \end{equation}
We next observe that $|\sigma F| \leq |F|$ for any tensor $F$ (this can be seen by using an orthonormal basis $\{ \eps^{i_1} \otimes \ldots \otimes \eps^{i_m} \}$ for $m$-tensors, Cauchy-Schwarz and the definitions), and $|F \otimes g| = n^{1/2} |F|$ (which also follows from the definitions). Thus $|\alpha F| \leq n^{1/2} |F|$. Consequently, using that the map $\lambda_x$ in Section \ref{subsec:tensors} is an isometry up to a factor depending on $n$ and $m$, 
\[
|h(x)|^2 \leq C_{n,m} \sum_{j=0}^{(m-1)/2} |U_{m-1-2j}(x)|^2 \leq C_{n,m} \sum_{j=0}^{(m-1)/2} \| u_{m-1-2j}(x,\,\cdot\,) \|_{L^2(S_x M)}^2.
\]
The orthogonality of spherical harmonics and the estimate \eqref{u_estimate_epsilon} imply that 
\[
|h(x)|^2 \leq C_{n,m} \int_{S_x M} |u(x,v)|^2 \,dS \leq C_{\eps,n,m} \e^{-2(\eta-\eps) d_g(x,o)}.
\]
This shows that $h \in E_{\eta-\eps}(M)$ as required. The proof in the case of Theorem \ref{thm_main2} follows similarly by replacing \eqref{u_estimate_epsilon} with the estimate in Lemma \ref{lma:uf-estimate}(b).
\end{proof}

\bibliographystyle{alpha}
\bibliography{grt_noncompact_ndim}{}

\end{document}